\documentclass[review,a4page]{elsarticle}
\usepackage[utf8]{inputenc}
\usepackage{graphicx}
\usepackage{srcltx}
\usepackage{eurosym}
\usepackage{mathtools}
\allowdisplaybreaks
\usepackage{enumerate}
\usepackage{amsmath}
\usepackage{amsfonts}
\usepackage{amssymb}
\usepackage{amsthm}
\usepackage{graphicx}
\usepackage{mathrsfs}
\usepackage{xcolor}
\usepackage{exscale}
\usepackage{latexsym}
\usepackage{esvect}
\usepackage[T1]{fontenc}

\usepackage{calc}
\usepackage[titletoc,toc,page]{appendix}
\usepackage[colorlinks,plainpages=true,pdfpagelabels,hypertexnames=true,colorlinks=true,pdfstartview=FitV,linkcolor=blue,citecolor=red,urlcolor=black]{hyperref}
\PassOptionsToPackage{unicode}{hyperref}
\PassOptionsToPackage{naturalnames}{hyperref}
\AtBeginDocument{%
\hypersetup{citecolor=red}}
\usepackage{enumerate}
\usepackage[shortlabels]{enumitem}
\usepackage{bookmark}
\usepackage{wasysym}
\usepackage{esint}
\usepackage{setspace}
\usepackage[ddmmyyyy]{datetime}
\numberwithin{equation}{section}
\everymath{\displaystyle}
\usepackage[margin=2.2cm]{geometry}

\usepackage[capitalize,nameinlink]{cleveref}
\crefname{section}{Section}{Sections}
\crefname{subsection}{Subsection}{Subsections}
\crefname{condition}{Condition}{Conditions}
\crefname{hypothesis}{Hypothesis}{Hypothesis}
\crefname{assumption}{Assumption}{Assumptions}
\crefname{lemma}{Lemma}{Lemmas}
\crefname{claim}{Claim}{Claims}
\crefname{remark}{Remark}{Remarks}

\crefformat{equation}{\textup{#2(#1)#3}}
\crefrangeformat{equation}{\textup{#3(#1)#4--#5(#2)#6}}
\crefmultiformat{equation}{\textup{#2(#1)#3}}{ and \textup{#2(#1)#3}}
{, \textup{#2(#1)#3}}{, and \textup{#2(#1)#3}}
\crefrangemultiformat{equation}{\textup{#3(#1)#4--#5(#2)#6}}%
{ and \textup{#3(#1)#4--#5(#2)#6}}{, \textup{#3(#1)#4--#5(#2)#6}}%
{, and \textup{#3(#1)#4--#5(#2)#6}}

\Crefformat{equation}{#2Equation~\textup{(#1)}#3}
\Crefrangeformat{equation}{Equations~\textup{#3(#1)#4--#5(#2)#6}}
\Crefmultiformat{equation}{Equations~\textup{#2(#1)#3}}{ and \textup{#2(#1)#3}}
{, \textup{#2(#1)#3}}{, and \textup{#2(#1)#3}}
\Crefrangemultiformat{equation}{Equations~\textup{#3(#1)#4--#5(#2)#6}}%
{ and \textup{#3(#1)#4--#5(#2)#6}}{, \textup{#3(#1)#4--#5(#2)#6}}%
{, and \textup{#3(#1)#4--#5(#2)#6}}

\crefdefaultlabelformat{#2\textup{#1}#3}
\newtheorem{theorem}{Theorem}[section]
\newtheorem{lemma}[theorem]{Lemma}

\newtheorem{proposition}[theorem]{Proposition}

\newtheorem{definition}[theorem]{Definition}% Use {\rm ...}
\newtheorem{remark}[theorem]{Remark}        % Use {\rm ...}

\numberwithin{equation}{section}
\def\YYint#1#2#3{{\setbox0=\hbox{$#1{#2#3}{\iint}$}
\vcenter{\hbox{$#2#3$}}\kern-.50\wd0}}
\def\XXint#1#2#3{{\setbox0=\hbox{$#1{#2#3}{\int}$}
\vcenter{\hbox{$#2#3$}}\kern-.50\wd0}}

\makeatletter
\def\namedlabel#1#2{\begingroup
\def\@currentlabel{#2}%
\label{#1}\endgroup
}
\makeatother
\makeatletter
\newcommand{\rmh}[1]{\mathpalette{\raisem@th{#1}}}
\newcommand{\raisem@th}[3]{\hspace*{-1pt}\raisebox{#1}{$#2#3$}}
\makeatother

\newcommand{\descref}[2]{\hyperref[#1]{\textcolor{black}{(}\textcolor{blue}{\bf #2}\textcolor{black}{)}}}
\newcommand{\dref}[2]{\hyperref[#1]{\textcolor{black}{(}\textcolor{blue}{\bf #2}\textcolor{black}{)}}}
\makeatletter
\g@addto@macro\normalsize{%
\setlength\abovedisplayskip{3pt}
\setlength\belowdisplayskip{3pt}
\setlength\abovedisplayshortskip{1pt}
\setlength\belowdisplayshortskip{3pt}
}
\makeatletter
\def\ps@pprintTitle{%
\let\@oddhead\@empty
\let\@evenhead\@empty
\def\@oddfoot{}%
\let\@evenfoot\@oddfoot}
\makeatother

\newcommand{\ep}{\epsilon}

\newcommand{\iprod}[2]{\langle #1 \ ,  #2\rangle}

\newcounter{whitney}
\refstepcounter{whitney}

\newcounter{ineqcounter}
\refstepcounter{ineqcounter}

\begin{document}
\begin{frontmatter}
%\addbibresource{main.bib}
\title{Multiplicity of solutions for mixed local-nonlocal elliptic equations with singular nonlinearity}

\author{Kaushik Bal and Stuti Das}
\ead{kaushik@iitk.ac.in and stutid21@iitk.ac.in}

\address{Department of Mathematics and Statistics,\\ Indian Institute of Technology Kanpur, Uttar Pradesh, 208016, India}

\newcommand*{\avint}{\mathop{\, \rlap{--}\!\!\int}\nolimits}

\begin{abstract}
We will prove multiplicity results for the mixed local-nonlocal elliptic equation of the form
\begin{eqnarray}
\begin{split}
-\Delta_pu+(-\Delta)_p^s u&=\frac{\lambda}{u^{\gamma}}+u^r  \text { in } \Omega, \\u&>0 \text{ in } \Omega,\\u&=0  \text { in }\mathbb{R}^n \backslash \Omega;
\end{split}
\end{eqnarray}
where 
\begin{equation*}
(-\Delta )_p^s u(x)= c_{n,s}\operatorname{P.V.}\int_{\mathbb{R}^n}\frac{|u(x)-u(y)|^{p-2}(u(x)-u(y))}{|x-y|^{n+sp}} d y, 
\end{equation*}
and $-\Delta_p$ is the usual $p$-Laplace operator.
Under the assumptions that $\Omega$ is a bounded domain 
in $\mathbb{R}^{n}$ with regular enough boundary, $p>1$, $n> p$, $s\in(0,1)$, $\lambda>0$ and $r\in(p-1,p^*-1)$ where $p^*$ is the critical Sobolev exponent, we will show there exist at least two weak solutions to our problem for $0<\gamma<1$ and some certain values of $\lambda$. Further, for every $\gamma>0$, assuming strict convexity of $\Omega$, for $p=2$ and $s\in(0,1/2)$, we will show the existence of at least two positive weak solutions to the problem, for small values of $\lambda$, extending the result of \cite{garaingeometric}. Here $c_{n,s}$ is a suitable normalization constant, and $\operatorname{P.V.}$ stands for Cauchy Principal Value.
\end{abstract}

%\begin{keyword}

%\MSC [2020]:
%\end{keyword}

\end{frontmatter}

\begin{singlespace}
\tableofcontents
\end{singlespace}

\section{Introduction}In this article, we deal with the multiplicity of weak solutions to the singular elliptic problems given by
\begin{equation}{\label{p1}}
\begin{split}
-\Delta_pu+(-\Delta)_p^s u&=\frac{\lambda}{u^{\gamma}}+ u^r  \text { in } \Omega, \\u&>0 \text{ in } \Omega,\\u&=0  \text { in }\mathbb{R}^n \backslash \Omega,
\end{split}
\end{equation}
where $\gamma\in(0,1)$, $n>p>1$, $s\in(0,1)$, $\lambda>0$, $p<r+1<p^*$, $\Omega$ is bounded domain in $\mathbb{R}^n$ with $C^1$ boundary;
and \begin{equation}{\label{p2}}
\begin{split}
-\Delta u+(-\Delta)^s u&=\frac{\lambda}{u^{\gamma}}+u^r  \text { in } \Omega, \\u&>0 \text{ in } \Omega,\\u&=0  \text { in }\mathbb{R}^n \backslash \Omega,
\end{split}
\end{equation}
for $\gamma\geq 1$. Of course, in the latter case, we need strict convexity and smooth boundary of $\Omega$ along with $0<s<1/2$.
To this aim, we start with a brief background of the problems available in the literature. \smallskip\\
Singular elliptic problems have been extensively studied for the past few decades, starting with the pioneering work of Crandall-Rabinowitz-Tartar \cite{CrRaTa}, who showed that the unperturbed case of \eqref{p2} (and for local operator), under Dirichlet boundary conditions given by
\begin{equation}{\label{pb}}
\begin{array}{c}
    -\Delta u=\frac{f}{u^{\gamma}}  \text { in } \Omega, \\
 u>0\;  \text { in } \Omega, \quad
u=0\;  \text { in } \partial \Omega,
\end{array}
\end{equation}
admits a unique solution $u\in C^2(\Omega)\cap C(\overline{\Omega})$ for any $\gamma>0$ along with the fact that the solution behaves like a distance function near the boundary provided $f$ is H\"older Continuous. Interestingly enough, Lazer-Mckenna \cite{LaMc} showed that the unique solution obtained by \cite{CrRaTa} is indeed in $W_0^{1,2}(\Omega)$ if and only if $0<\gamma<3$. They also showed that the solution belongs to $C^1(\overline{\Omega})$ provided $0<\gamma<1$. This was followed for the perturbed singular case by the work of Haitao \cite{Haitao}, who studied the problem
\begin{equation}{\label{pb2}}
\begin{array}{c}
    -\Delta u=\frac{\lambda}{u^{\gamma}}+u^r  \text { in } \Omega,\\
 u>0\;  \text { in } \Omega, \quad
u=0\;  \text { in } \partial \Omega,
\end{array}
\end{equation}
and showed the existence of $\Lambda>0$ such that there exists at least two solutions $u, v \in W_0^{1,2}(\Omega)$ to problem \eqref{pb2} for $\lambda<\Lambda$, no solution for $\lambda>\Lambda$ and at least one solution for $\lambda=\Lambda$ provided $0<\gamma<1<r \leq 2^*-1$. The generalization of these results for $p$-Laplacian was given by Giacomoni et al. \cite{Giacomoni}, who showed the existence of at least two solutions for $0<\gamma<1$ and $p-1<r\leq p^*-1$. In the above-mentioned works on the perturbed problems, the solution so obtained satisfied the boundary condition in the trace sense, and the restriction $0<\gamma<1$ is due to the use of variational methods which require the associated functional to be well-defined on $W_0^{1, p}(\Omega)$.
\smallskip\\Boccardo-Orsina \cite{orsina} in a beautiful paper showed the followings regarding solutions of \eqref{pb} 
\begin{equation*}
    \begin{cases}u \in W_0^{1, \frac{n m(1+\gamma)}{n-m(1-\gamma)}}(\Omega) & \text { if } 0<\gamma<1 \text { and } f \in L^m(\Omega) \text { with } m \in\left[1,\left(2^* /(1-\gamma)\right)^{\prime}\right), \\ u \in W_0^{1,2}(\Omega) & \text { if } 0<\gamma<1 \text { and } f \in L^m(\Omega) \text { with } m=\left(2^* /(1-\gamma)\right)^{\prime}, \\ u \in W_0^{1,2}(\Omega) & \text { if } \gamma=1 \text { and } f \in L^1(\Omega), \\ u^{\frac{1+\gamma}{2}} \in W_0^{1,2}(\Omega) & \text { if } \gamma>1 \text { and } f \in L^1(\Omega).\end{cases}
\end{equation*}The boundary condition is now understood as such that $u^{\frac{1+\gamma}{2}}$ belongs to $W_0^{1,2}(\Omega)$. This has been generalized by Canino et al. \cite{Canino} for the $p$-Laplacian where existence of a solution $u \in W_{\operatorname {loc }}^{1, p}(\Omega)$ was shown for $\gamma>0$ and $f \in L^1(\Omega)$ such that $u^{\frac{p-1+\delta}{p}} \in W_0^{1, p}(\Omega)$. The perturbed problem \eqref{pb2} was studied by Arcoya-Boccardo \cite{Arcoya} for $0<\gamma< 1$ using the variational method. Again for $\gamma\geq 1$, Arcoya-Mérida \cite{Arcoyamultigreaterthan} obtained the existence of at least two solutions in $W_{\operatorname{loc}}^{1,2}(\Omega) \cap L^{\infty}(\Omega)$. Moreover, any solution $u$ so obtained satisfies $u^{\frac{1+\gamma}{2}} \in W_0^{1,2}(\Omega)$. This was generalized for the quasilinear case by Bal-Garain \cite{quasilinear} for any $\frac{2 n+2}{n+2}<p<n$.
\smallskip\\ For the nonlocal case of \cref{pb}, we refer \cite{fracsingular,scase}, where the authors obtained, among other results, existence and summability of weak solutions. For multiplicity results in nonlocal perturbed singular case, with $\gamma<1$, one can see \cite{fracvarisin} and the references therein. As for the mixed local-nonlocal problem, the literature is very little known. Recently, in \cite{arora,garain}, the authors have studied the singular problems associated with mixed operators given by
\begin{equation*}
\begin{array}{c}
-\Delta_pu+(-\Delta)_p^s u=\frac{f}{u^{\gamma}}  \text { in } \Omega, \smallskip\\u>0 \text{ in } \Omega,\quad u=0  \text { in }\mathbb{R}^n \backslash \Omega;
\end{array}
\end{equation*}
and obtained plenty of results regarding existence and other properties of solutions. In \cite{garaingeometric}, the author has used variational techniques and shown there exist at least two solutions to \eqref{p2} for $\gamma<1$. We will generalize this for the quasilinear case \cref{p1}. Further, we also aim to get multiplicity results similar to Arcoya-Mérida \cite{Arcoyamultigreaterthan} in mixed local nonlocal setting for \cref{p2}.\smallskip\\ In this article, we will extensively use the regularity results, maximum principles and other properties obtained in \cite{Biagiregularitymaximum,FaberKrahn,eigenvalue,Biagisymmetry,MG,valdinoci}. We finish our literature survey by providing some references for parabolic problems as \cite{abdellaoui2,bal2024mixed,parabolic1}. Now that the history of the problem is clear, let us discuss the difficulties one encounters while studying the problems \cref{p1,p2} and the strategy we employ to circumvent those difficulties.\smallskip\\
\textbf{Goal of this paper:}
We intend to deduce the existence of two different weak solutions to \cref{p1}, with $1<p<\infty$ and for $0<\gamma<1$. For $\gamma\geq1$, we continue to show the existence of at least two weak solutions of the semilinear case of \cref{p1} given by \cref{p2}. \smallskip\\\textbf{Difficulties:} To deal with the singular term, we follow the classical approach of obtaining solutions to a sequence of approximated problems given by
\begin{equation}{\label{appro}}
\begin{array}{c}
-\Delta_pu+(-\Delta)_p^s u=\frac{\lambda}{(u+\frac{1}{k})^{\gamma}}+u^r  \text { in } \Omega, \smallskip\\u>0 \text{ in } \Omega,\quad u=0  \text { in }\mathbb{R}^n \backslash \Omega,
\end{array}
\end{equation}
where $k\in \mathbb{N}$. Note that this problem is non-singular for any $k \in \mathbb{N}$. We start by showing the existence of two different solutions to \cref{appro} for each $k\in\mathbb{N}$. For \cref{p1}, we will use the variational technique (as $\gamma<1$ allows us so) and show that the functional corresponding to \cref{appro} satisfies the mountain pass geometry. We will obtain a solution as the critical point of the functional which is different from the minimizer. This, along with a uniform apriori bound in $W^{1,p}_0(\Omega)$, will give our multiplicity result (of course, one needs to show the almost everywhere convergence of the gradients of weak solutions (of \cref{appro}) to the gradient of their limit, to pass $k\to\infty$ in \cref{appro}).\smallskip\\For \cref{p2}, we obtain a uniform apriori estimate for the $L^\infty$ bound of the solutions of \cref{appro}  (with $p=2$) independent of $k$ and then use Leray-Schauder degree. We conclude by passing to the limit to obtain two distinct solutions to our main problem. One of this study's main challenges is finding the uniform a priori estimates independent of $k$. For the Laplacian operator, one can use Kelvin transformation (see \cite{Arcoyamultigreaterthan}) which fails for the mixed local-nonlocal case. We overcome this difficulty by considering strict convexity of our domain. Once we have a uniform neighbourhood of the boundary, the blow-up analysis of Gidas-Spruck \cite{GidasSpruck} goes through. Here we consider the fractional Laplace operator as a lower order term and approach by taking limit. \\
Another difficulty arises in constructing an appropriate subsolution to the approximated problem for the unperturbed case, which is handled by taking the first eigenfunction of Laplacian and using its regularity. It is important to mention that we also prove existence result regarding \cref{p2} in \cref{th2}, which represents the boundary data of $u$. A result regarding nonexistence of solutions to \cref{p2} for large $\lambda$ is also obtained.
\subsection*{\textbf{Organization of the article}} In \cref{prelims}, we will write notations and give definitions and embedding results regarding Sobolev spaces we need. Appropriate notions of weak solutions for our problems will be defined, and the main results will be stated. The next section will contain preliminaries required for \cref{p1}, and the multiplicity result regarding this will be proved in \cref{th1}. \smallskip\\ \cref{th2} contains an existence result to guarantee the boundary behaviour of solutions to \cref{p2}. The next section is devoted to the preliminaries of \cref{p2}. One will get the existence of two different uniformly bounded (in $L^\infty$) sequences of solutions to the approximated problems here, and these will be used in \cref{th3} to show the multiplicity result. We will end our discussion with a related problem not involving perturbation following \cite{Arcoyamultigreaterthan}.
\section{Preliminaries}{\label{prelims}}
\subsection{\textbf{Notations}} We gather here all the standard notations that will be used throughout the paper.\smallskip\\
$\bullet$ We will take $n$ to be the space dimension and $\Omega$ be an open bounded domain in $\mathbb{R}^n$ with $C^1$ or smooth boundary.\smallskip\\
$\bullet$ For $q>1$, the H\"older conjugate exponent of $q$ will be denoted by $q^\prime=\frac{q}{q-1}$.\smallskip\\
$\bullet$ The Lebesgue measure of a measurable subset $\mathrm{S}\subset \mathbb{R}^n$ will be denoted by $|\mathrm{S}|$.\smallskip\\
$\bullet$ For any open subset $\Omega$ of $\mathbb{R}^n$, $K\subset\subset \Omega $ will imply $K$ is compactly contained in $\Omega.$\smallskip\\
%$\bullet$ We shall use the notations for the balls and parabolic cylinders as
%\begin{equation*}
%\begin{array}{ll}
%B_\varrho\left(x_0\right)=\left\{x \in \mathbb{R}^n:\left|x-x_0\right|<\varrho\right\}, & \bar{B}_\varrho\left(x_0\right)=\left\{x \in \mathbb{R}^n:\left|x-x_0\right| \leq \varrho\right\}, \\
%I_\varrho\left(t_0\right)=\left\{t \in \mathbb{R}: t_0-\varrho^2<t<t_0\right\}, & Q_{\varrho}\left(z_0\right)=B_\varrho\left(x_0\right) \times I_\varrho\left(t_0\right) .
%\end{array}
%\end{equation*}
$\bullet$ $\int$ will denote integration concerning space only, and integration on $\Omega \times \Omega$ or $\mathbb{R}^n \times \mathbb{R}^n$ will be denoted by a double integral $\iint$. Moreover, average integral will be denoted by $\fint$.\smallskip\\
$\bullet$ The notation $a \lesssim b$ will be used for $a \leq C b$, where $C$ is a universal constant which only depends on the dimension $n$ and sometimes on $s$ too. $C$ (or sometimes $c$) may vary from line to line or even in the same line.\smallskip\\
$\bullet$ For a function $h$, we denote its positive and negative parts by $h^+=\max\{h,0\}$, $h^-=\max\{-h,0\}$ respectively.%\smallskip\\
%$\bullet$ For $k\in \mathbb{N}$, we denote $T_k(\sigma)=\max \{-k, \min \{k, \sigma\}\}$, for $\sigma \in \mathbb{R}$.
\subsection{\textbf{Function Spaces}}
In this section, we present 
 definitions and properties of some function spaces that will be useful for our work. We recall that for $E \subset \mathbb{R}^n$, the Lebesgue space
$L^p(E), 1 \leq p<\infty$, is defined to be the space of $p$-integrable functions $u: E \rightarrow \mathbb{R}$ with the finite norm
\begin{equation*}
\|u\|_{L^p(E)}=\left(\int_E|u(x)|^p d x\right)^{1 / p} .
\end{equation*}
By $L_{\operatorname{loc }}^p(E)$ we denote the space of locally $p$-integrable functions, which means, $u \in L_{\operatorname{loc }}^p(E)$ if and only if $u \in L^p(F)$ for every $F \subset\subset E$. In the case $0<p<1$, we denote by $L^p(E)$ a set of measurable functions such that $\int_E|u(x)|^p d x<\infty$.
\begin{definition}
    The Sobolev space $W^{1, p}(\Omega)$, for $1 \leq p<\infty$, is defined as the Banach space of locally integrable weakly differentiable functions $u: \Omega \rightarrow \mathbb{R}$ equipped with the following norm
\begin{equation*}
\|u\|_{W^{1, p}(\Omega)}=\|u\|_{L^p(\Omega)}+\|\nabla u\|_{L^p(\Omega)} .
\end{equation*}
\end{definition}
The space $W_0^{1, p}(\Omega)$ is defined as the closure of the space ${C}_c^{\infty}(\Omega)$, in the norm of the Sobolev space $W^{1, p}(\Omega)$, where ${C}^\infty_c(\Omega)$ is the set of all smooth functions whose supports are compactly contained in $\Omega$.
%We now proceed with defining the fractional Sobolev spaces. We will also state the embedding results regarding those spaces.  
\begin{definition}
    Let $0<s<1$ and $\Omega$ be an open connected subset of $\mathbb{R}^n$ with $C^1$ boundary. The fractional Sobolev space $W^{s, q}(\Omega)$ for any $1\leq q<+\infty$ is defined by
\begin{equation*}
    W^{s, q}(\Omega)=\left\{u \in L^q(\Omega): \frac{|u(x)-u(y)|}{|x-y|^{\frac{n}{q}+s}} \in L^q(\Omega\times\Omega)\right\},
\end{equation*}
and it is endowed with the norm
\begin{equation}{\label{norm}}
\|u\|_{W^{s, q}(\Omega)}=\left(\int_{\Omega}|u(x)|^q d x+\int_{\Omega} \int_{\Omega} \frac{|u(x)-u(y)|^q}{|x-y|^{n+sq}}d x d y\right)^{1/q}.
\end{equation}
\end{definition}
It can be treated as an intermediate space between $W^{1,q}(\Omega)$ and $L^q(\Omega)$. For $0<s\leq s^{\prime}<1$, $W^{s^{\prime},q}(\Omega)$ is continuously embedded in $W^{s,q}(\Omega)$, see [\citealp{frac}, Proposition 2.1]. The fractional Sobolev space with zero boundary values is defined by
\begin{equation*}
W_0^{s, q}(\Omega)=\left\{u \in W^{s, q}(\mathbb{R}^n): u=0 \text { in } \mathbb{R}^n \backslash \Omega\right\}.
\end{equation*}
However $W_0^{s, q}(\Omega)$ can be treated as the closure of $ {C}^\infty_c(\Omega)$ in $W^{s,q}(\Omega)$ with respect to the fractional Sobolev norm defined in \cref{norm}. Both $W^{s, q}(\Omega)$ and $W_0^{s, q}(\Omega)$ are reflexive Banach spaces, for $q>1$, for details we refer to the readers [\citealp{frac}, Section 2]. The spaces $W^{1, p}(\Omega)$ and $W_0^{1, p}(\Omega)$ are also reflexive for $p>1$.\smallskip\\
The following result asserts that the classical Sobolev space is continuously embedded in the fractional Sobolev space; see [\citealp{frac}, Proposition 2.2]. The idea applies an extension property of $\Omega$ so that we can extend functions from $W^{1,q}(\Omega)$ to $W^{1,q}(\mathbb{R}^n)$ and that the extension operator is bounded.
\begin{lemma}{\label{embedding}}
    Let $\Omega$ be a bounded domain in $\mathbb{R}^n$ with $C^{0,1}$ boundary. Then $\exists\, C=C(\Omega, n, s)>0$ such that
\begin{equation*}
\|u\|_{W^{s, q}(\Omega)} \leq C\|u\|_{W^{1,q}(\Omega)},
\end{equation*}
for every $u \in W^{1,q}(\Omega)$.
\end{lemma}
For the fractional Sobolev spaces with zero boundary value, the next embedding result follows from [\citealp{frac2}, Lemma 2.1]. The fundamental difference of it compared to \cref{embedding} is that the result holds for any bounded domain (without any smoothness condition on the boundary), since for the Sobolev spaces with zero boundary value, we always have a zero extension to the complement.
\begin{lemma}{\label{embedding2}} Let $\Omega$ be a bounded domain in $\mathbb{R}^n$ and $0<s<1$. Then $\exists\, C=C(n, s, \Omega)>0$ such that
\begin{equation*}
\int_{\mathbb{R}^n} \int_{\mathbb{R}^n} \frac{|u(x)-u(y)|^q}{|x-y|^{n+sq}} d x d y \leq C \int_{\Omega}|\nabla u|^qd x,
\end{equation*}
for every $u \in W_0^{1,q}(\Omega)$. Here, we consider the zero extension of $u$ to the complement of $\Omega$.
\end{lemma}
We now proceed with the basic Poincar\'{e} inequality, which can be found in [\citealp{LCE}, Chapter 5, Section 5.8.1].
\begin{lemma}{\label{p}}
  Let $\Omega\subset \mathbb{R}^n$ be a bounded domain with $ {C}^1$ boundary and $q\geq 1$. Then there exists a positive constant $C>0$ depending only on $n$ and $ \Omega$, such that \begin{equation*} 
  \int_\Omega |u|^q d x\leq C\int_\Omega |\nabla u|^q d x, \qquad\forall u\in W^{1,q}_0(\Omega).
  \end{equation*}
  Specifically if we take $\Omega=B_{\bar r}$, then we will get for all $u\in W^{1,q}(B_{\bar r})$,
  \begin{equation*}
  \fint_{B_{\bar r}}\left|u-(u)_{B_{\bar r}}\right|^q d x \leq c {\bar r}^q \fint_{B_{\bar r}} |\nabla u|^qd x,
  \end{equation*}
  where $c$ is a constant depending only on $n$, and $(u)_{B_{\bar r}}$ denotes the average of $u$ in $B_{\bar r}$, and $B_{\bar r}$ denotes a ball of radius ${\bar r}$ centered at $x_0\in \mathbb{R}^n$.% Here, $\fint$ denotes the average integration.
 \end{lemma}
Using \cref{embedding2}, and the above Poincar\'e inequality, we observe that the following norm on the space $W^{1,q}_0(\Omega)$ defined by 
 \begin{equation*}
\|u\|_{W^{1,q}_0(\Omega)}=\left(\int_\Omega |\nabla u|^q d x +\int_{\mathbb{R}^n} \int_{\mathbb{R}^n} \frac{|u(x)-u(y)|^q}{|x-y|^{n+sq}} d x d y \right)^{1/q},%{\frac{1}{q}},
\end{equation*}
is equivalent to the norm
 \begin{equation*}
\|u\|_{W^{1,q}_0(\Omega)}=\left(\int_\Omega |\nabla u|^q d x  \right)^{1/q}.%{\frac{1}{q}} .     
 \end{equation*}
The following is a version of fractional Poincar\'{e}.
\begin{lemma}{\label{fracpoin}}
Let $\Omega\subset \mathbb{R}^n$ be a bounded domain with $ {C}^1$ boundary, $s \in(0,1)$ and $q\geq 1$. If $u \in W^{s,q}_0(\Omega)$, then
\begin{equation*}
\int_\Omega |u|^q d x \leq c\int_{\Omega} \int_{\Omega} \frac{|u(x)-u(y)|^q}{|x-y|^{n+sq}} d x d y ,
\end{equation*}
holds with $c \equiv c(n, s,\Omega)$.
\end{lemma}
In view of \cref{fracpoin}, we observe that the Banach space $W_0^{s, q}(\Omega)$ can be endowed with the norm
\begin{equation*}
\|u\|_{W_0^{s, q}(\Omega)}=\left(\int_{\Omega} \int_{\Omega} \frac{|u(x)-u(y)|^q}{|x-y|^{n+sq}} d x d y\right)^{1/q},%{\frac{1}{q}},
\end{equation*}
which is equivalent to that of $\|u\|_{W^{s, q}(\Omega)}$. 
Now, we define the local spaces as
\begin{equation*}
     W_{\operatorname{loc }}^{1, q}(\Omega)=\left\{u: \Omega \rightarrow \mathbb{R}: u \in L^q(K), \int_K |\nabla u|^q d x<\infty, \text { for every } K \subset \subset \Omega\right\} ,
\end{equation*}
and 
\begin{equation*}
     W_{\operatorname{loc }}^{s, q}(\Omega)=\left\{u: \Omega \rightarrow \mathbb{R}: u \in L^q(K), \int_K \int_K \frac{|u(x)-u(y)|^q}{|x-y|^{n+ sq}} d x d y<\infty, \text { for every } K \subset \subset \Omega\right\} .
\end{equation*}
Now for $n>p$, we define the critical Sobolev exponent as $p^*=\frac{np}{n-p}$, then we get the following embedding result for any bounded open subset $\Omega$ of class $ {C}^1$ in $\mathbb{R}^n$, see for details [\citealp{LCE}, Chapter 5].
\begin{theorem}{\label{Sobolev embedding}} Let $n>p$. Then, $\exists\, C\equiv C(n,\Omega)>0$, such that for all $u \in  {C}_c^{\infty}(\Omega)$
\begin{equation*}
    \|u\|_{L^{p^*}(\Omega)}^p \leq C \int_{\Omega} |\nabla u|^p d x.
\end{equation*}
Moreover the inclusion map
\begin{equation*}
W_0^{1, p}(\Omega) \hookrightarrow L^r(\Omega)    
\end{equation*}
is continuous for $1 \leq r \leq p^*$ and the above embedding is compact except for $r=p^*$.
\end{theorem}
Similarly, for $n>ps$, we define the fractional Sobolev critical exponent as $p_s^*=\frac{ np}{n-p s}$. The following result is a fractional version of the Sobolev inequality (\cref{Sobolev embedding}) which also implies a continuous embedding of $W_0^{s,p}(\Omega)$ in the critical Lebesgue space $L^{p_s^*}(\Omega)$. One can see the proof in \cite{frac}.
 \begin{theorem}{\label{Fractional Sobolev embedding}} Let $0<s<1$ be such that $n>p s$. Then, there exists a constant $S(n, s)$ depending only on $n$ and $s$, such that for all $u \in  {C}_c^{\infty}(\Omega%\mathbb{R}^n
 )$
\begin{equation*}
\|u\|_{L^{p_s^*}(\Omega%\mathbb{R}^n
)} \leq S(n, s) \left(\int_{\Omega}\int_{\Omega%\mathbb{R}^{2 n}
} \frac{|u(x)-u(y)|^p}{|x-y|^{n+sp}} d x d y \right)^{1/p}.    
\end{equation*}
\end{theorem}
We now recall the following algebraic inequality that can be found in [\citealp{abdellaoui1}, Lemma 2.22].% and will be used extensively.
\begin{lemma}{\label{algebraic}}
\begin{enumerate}[label=\roman*)]
    \item Let $\alpha>0$. For every $x, y \geq 0$ one has
\begin{equation*}
    (x-y)(x^\alpha-y^\alpha) \geq \frac{4 \alpha}{(\alpha+1)^2}\left(x^{\frac{\alpha+1}{2}}-y^{\frac{\alpha+1}{2}}\right)^2.
\end{equation*}
\item Let $0<\alpha \leq 1$. For every $x, y \geq 0$ with $x \neq y$ one has
\begin{equation*}
    \frac{x-y}{x^\alpha-y^\alpha} \leq \frac{1}{\alpha}(x^{1-\alpha}+y^{1-\alpha}).
\end{equation*}
\item Let $\alpha \geq 1$. Then there exists a constant $C_\alpha$ depending only on $\alpha$ such that
\begin{equation*}
|x+y|^{\alpha-1}|x-y| \leq C_\alpha\left|x^\alpha-y^\alpha\right| .   
\end{equation*}    
\end{enumerate}
\end{lemma}
Next, we state the algebraic inequality from [\citealp{ineq}, Lemma 2.1].
\begin{lemma}{\label{p case}} Let $1<p<\infty$. Then for any $a,b\in \mathbb{R}^n$, there exists a constant $C\equiv C(p)>0$, such that
\begin{equation*}
    \iprod{|a|^{p-2}a-|b|^{p-2}b}{a-b}\geq C\frac{|a-b|^2}{(|a|+|b|)^{2-p}}.
\end{equation*}
\end{lemma}
\subsection{\textbf{Weak Solutions and Main Results}}
We are interested in multiplicity of weak solutions to the problem 
\begin{equation}{\label{prob}}
\begin{split}
-\Delta_pu+(-\Delta)_p^s u&=\frac{\lambda}{u^{\gamma}}+u^r  \text { in } \Omega, \\u&>0 \text{ in } \Omega,\\u&=0  \text { in }\mathbb{R}^n \backslash \Omega;
\end{split}
\end{equation}
where $n>p>1$; $\lambda>0$, $0<\gamma<1$; $r\in(p-1,p^*-1).$
%We split two cases $0<\gamma<1$ and $\gamma>1$. For $0<\gamma<1$, we define:
\begin{definition}{\label{weak sol}}
    A function $u\in W^{1,p}_0(\Omega)$ is said to be a weak solution to \cref{prob} if $u>0$ in $\Omega$ such that for every $\omega\subset\subset \Omega$, there exists a positive constant $c(\omega)$ with $u\geq c(\omega)>0$ in $\omega$ and for all $\phi\in  {C}_c^\infty(\Omega)$, we have 
    \begin{equation*}
    \begin{array}{c}
        \int_\Omega|\nabla u|^{p-2}\nabla u\cdot\nabla \phi \,d x+\int_{\mathbb{R}^n}\int_{\mathbb{R}^n}\frac{|u(x)-u(y)|^{p-2}(u(x)-u(y))(\phi(x)-\phi(y))}{|x-y|^{n+sp}} d x d y\\=\lambda\int_\Omega \frac{\phi}{u^\gamma} d x +\int_\Omega u^r\phi \,d x .
        \end{array}
    \end{equation*}
\end{definition}
Further, consider the problem
\begin{equation}{\label{prob2}}
\begin{split}
-\Delta u+(-\Delta)^s u&=\frac{\lambda}{u^{\gamma}}+u^r  \text { in } \Omega, \\u&>0 \text{ in } \Omega,\\u&=0  \text { in }\mathbb{R}^n \backslash \Omega,
\end{split}
\end{equation}
where %$n>2$, 
$\lambda,\gamma>0$, $r>1%<2^*-1
$. We define the weak solution as:
\begin{definition}{\label{weaksol2}}
    A function $u \in W_{\operatorname{loc}}^{1,2}(\Omega) \cap L^{\infty}(\Omega)$ (with some suitable power of $u$ in $W^{1,2}_0(\Omega)$) such that $\frac{\phi}{u^\gamma} \in L^1(\Omega)$ for every $\phi \in W_0^{1,2}(\omega)$ is said to be a weak solution of \cref{prob2} if it satisfies
\begin{equation*}
\int_{\Omega} \nabla u \cdot\nabla \phi\, d x+\int_{\mathbb{R}^n}\int_{\mathbb{R}^n}\frac{(u(x)-u(y))(\phi(x)-\phi(y))}{|x-y|^{n+2s}} d x d y=\lambda\int_\Omega \frac{\phi}{u^\gamma} d x +\int_\Omega u^r\phi \,d x, %\quad \forall \phi \in W_0^{1,2}(\omega),
\end{equation*}
for all $\phi \in W_0^{1,2}(\omega)$ and for every open subset $\omega$ of $\Omega$, such that $\omega \subset \subset \Omega$.
\end{definition} 
The main results of this article are as follows;
\begin{theorem}{\label{mainth1}} Let $0<\gamma<1$ and $\Omega$ be a bounded domain with $ {C}^1$ boundary. Then there exists $\Lambda>0$ such that for all $\lambda\in(0,\Lambda),$ the problem \cref{prob}
admits at least two distinct weak solutions in the sense of \cref{weak sol}.
\end{theorem}
\begin{theorem}{\label{mainth2}}
   Let $\Omega$ be a $ {C}^1$ bounded open set in $\mathbb{R}^n$. Then $\exists\,\Lambda>0$, such that $\forall\,\lambda\in(0,\Lambda)$, the problem \cref{prob2} has a positive weak solution $u \in W_{\operatorname {loc }}^{1,2}(\Omega) \cap L^{\infty}(\Omega)$ with $u \in W_0^{1,2}(\Omega)$, if $0<\gamma \leq 1$ and $u^{\frac{\gamma+1}{2}} \in W_0^{1,2}(\Omega)$, if $\gamma>1$.
\end{theorem}
\begin{theorem}{\label{mainth3}}
    Let $\Omega$ be a bounded strictly convex domain in $\mathbb{R}^n$ with smooth boundary, $n>2$ and let $s\in(0,1/2)$. If $\gamma>0$ and $2<r+1<2^*$ holds, then there is $\Lambda>0$ such that for every $\lambda \in(0, \Lambda)$ the problem \cref{prob2} has two different strictly positive solutions $u$ and $v$ in $W_{\operatorname {loc}}^{1,2}(\Omega) \cap L^{\infty}(\Omega)$ in the sense of \cref{weaksol2} such that
\begin{equation*}
u^\alpha, v^\alpha \in W_0^{1,2}(\Omega), \quad \forall \alpha>\frac{\gamma+1}{4} .
\end{equation*}
\end{theorem}
\begin{remark}
    Note that in the case $1<\gamma<3$, we have $\frac{\gamma+1}{4}<1$ and thus we obtain that the solutions are in $W_0^{1,2}(\Omega)$ too, improving the result of \cref{mainth2}. Furthermore, since $\frac{\gamma+1}{4}<\frac{\gamma+1}{2}$, the theorem also improves the regularity in \cref{mainth2} for the case $\gamma>1$.
\end{remark}
\section*{Preliminaries for the proof of \cref{mainth1}} Following the ideas and approximation techniques of \cite{Arcoya}, we will prove our first result \cref{mainth1} in this section. We consider the space $W^{1,p}_0(\Omega)$ with the norm \begin{equation*}\|u\|_{W^{1,p}_0(\Omega)}=\left(\int_\Omega |\nabla u|^p d x +\int_{\mathbb{R}^n} \int_{\mathbb{R}^n} \frac{|u(x)-u(y)|^p}{|x-y|^{n+ s p}} d x d y \right)^{1/p}%{\frac{1}{p}}
.\end{equation*} Let us denote the energy functional $I_\lambda:  W^{1,p}_0(\Omega)\rightarrow \mathbb{R} \cup\{ \pm \infty\}$ corresponding to the problem \cref{prob} by
\begin{equation*}
I_\lambda(u)=\frac{1}{p}\int_\Omega |\nabla u|^p d x +\frac{1}{p} \int_{\mathbb{R}^{n}}\int_{\mathbb{R}^{n}} \frac{|u(x)-u(y)|^p}{|x-y|^{n+s p}} d x d y-\lambda \int_{\Omega} \frac{\left(u^{+}\right)^{1-\gamma}}{1-\gamma} d x-\frac{1}{r+1} \int_{\Omega}\left(u^{+}\right)^{r+1} d x .
\end{equation*}
Note that $I_\lambda$ is not differentiable. Therefore for $\epsilon>0$, we consider the following approximated problem
\begin{equation}{\label{approximated}}
\begin{split}
-\Delta_pu+(-\Delta)_p^s u&=\frac{\lambda}{(u^++\epsilon)^{\gamma}}+(u^+)^r  \text { in } \Omega, %\\u&>0 \text{ in } \Omega,
\\u&=0  \text { in }\mathbb{R}^n \backslash \Omega,
\end{split}
\end{equation}
for which the corresponding energy functional is given by
\begin{equation*}
\begin{array}{rcl}
I_{\lambda,\epsilon}(u)&=&\frac{1}{p}\int_\Omega |\nabla u|^p d x +\frac{1}{p} \int_{\mathbb{R}^{n}}\int_{\mathbb{R}^{n}} \frac{|u(x)-u(y)|^p}{|x-y|^{n+s p}} d x d y-\frac{\lambda}{1-\gamma} \int_{\Omega} [\left(u^{+}+\epsilon\right)^{1-\gamma}-\epsilon^{1-\gamma}] d x\smallskip\\&&-\frac{1}{r+1} \int_{\Omega}\left(u^{+}\right)^{r+1} d x .
\end{array}
\end{equation*}
One can easily verify that $I_{\lambda, \epsilon} \in C^1(W^{1,p}_0(\Omega), \mathbb{R}), I_{\lambda, \epsilon}(0)=0$ and $I_{\lambda, \epsilon}(v) \leq I_{0, \epsilon}(v)$ for all $0 \leq v \in W^{1,p}_0(\Omega)$. We refer [\citealp{eigenvalue}, Proposition 5.1] for the existence of the first nonnegative eigenfunction $e_1 \in W^{1,p}_0(\Omega)\cap L^{\infty}(\Omega)$ corresponding to the first eigenvalue $\lambda_1$ satisfying the equation
\begin{equation*}
\begin{array}{c}
-\Delta_p v+(-\Delta)_p^s v=\lambda_1|v|^{p-2} v \text { in } \Omega,\quad v=0 \text { in } \mathbb{R}^n \backslash \Omega .    \end{array}
\end{equation*}
Without loss of generality, let $\left\|e_1\right\|_{W^{1,p}_0(\Omega)}=1$. Now we show the functional $I_{\lambda, \epsilon}$ satisfies the Palais-Smale $(P S)_c$ condition.
\begin{proposition}{\label{ps}} $I_{\lambda, \epsilon}$ satisfies the $(P S)_c$ condition, for any $c \in \mathbb{R}$, that is if $\left\{u_k\right\} \subset W^{1,p}_0(\Omega)$ is a sequence which satisfies
\begin{equation}{\label{psc}}
I_{\lambda, \epsilon}(u_k) \rightarrow c \text { and } I_{\lambda, \epsilon}^{\prime}(u_k) \rightarrow 0
\end{equation}
as $k \rightarrow \infty$, then $\left\{u_k\right\}$ admits a strongly convergent subsequence in $W^{1,p}_0(\Omega)$.
\end{proposition}
\begin{proof} We claim that if $\left\{u_k\right\} \subset W^{1,p}_0(\Omega)$ satisfies \cref{psc} then $\left\{u_k\right\}$ is bounded in $W^{1,p}_0(\Omega)$. To show this, we proceed as follows. As $0<\gamma<1$, we have 
\begin{equation}{\label{fracineq}}
    (u^++\epsilon)^{1-\gamma}-\epsilon^{1-\gamma}\leq (u^+)^{1-\gamma}.
\end{equation}
Using \cref{Sobolev embedding} and \cref{fracineq}, it holds
\begin{equation}{\label{step1}}
\begin{array}{rcl}
 I_{\lambda, \epsilon}(u_k)-\frac{1}{r+1} I_{\lambda, \epsilon}^{\prime}(u_k) u_k&=&
\left(\frac{1}{p}-\frac{1}{r+1}\right)\left(\int_\Omega |\nabla u|^p d x +\int_{\mathbb{R}^n} \int_{\mathbb{R}^n} \frac{|u(x)-u(y)|^p}{|x-y|^{n+sp}} d x d y \right)%\left\|u_k\right\|^p
\smallskip\\&&-\frac{\lambda}{1-\gamma} \int_{\Omega}[\left(u_k^{+}+\epsilon\right)^{1-\gamma}-\epsilon^{1-\gamma} ]d x+\frac{\lambda}{r+1} \int_{\Omega}\left(u_k^{+}+\epsilon\right)^{-\gamma} u_k\, d x \smallskip\\
&\geq & \left(\frac{1}{p}-\frac{1}{r+1}\right)\left\|u_k\right\|^p-\frac{\lambda}{1-\gamma} \int_{\Omega}\left(u_k^{+}\right)^{1-\gamma} d x+\frac{\lambda}{r+1} \int_{\Omega}\left(u_k^{+}+\epsilon\right)^{-\gamma} u_k \,d x \smallskip\\
&\geq & \left(\frac{1}{p}-\frac{1}{r+1}\right)\left\|u_k\right\|^p-\frac{\lambda}{1-\gamma} \int_{\Omega}\left(u_k^{+}\right)^{1-\gamma} d x-\frac{\lambda C}{\epsilon(r+1)}\left\|u_k\right\|,
\end{array}
\end{equation}
where $C$ is a positive constant independent of $k$. Note that, for the last term, we have used the fact that if $a\in \mathbb{R}$ is such that $|a|<b$ then $-b<-|a|<a<|a|<b$ holds. Here $\|\cdot\|$ denotes norm in $W^{1,p}_0(\Omega)$. \\Further using \cref{Sobolev embedding}, we estimate the second term as

\begin{equation}{\label{step2}}
\begin{array}{rcl}
 \int_{\Omega} \left(u_k^{+}\right)^{1-\gamma} d x \leq\int_{\Omega} \left|u_k\right|^{1-\gamma}d x 
&\leq & \left(\int_{\Omega \cap\left\{\left|u_k\right| \geq 1\right\}}\left|u_k\right|^{1-\gamma} d x+\int_{\Omega \cap\left\{\left|u_k\right|<1\right\}}\left|u_k\right|^{1-\gamma} d x\right)\smallskip\\  
& \leq &\left(\int_{\Omega \cap\left\{\left|u_k\right| \geq 1\right\}}\left|u_k\right| d x+\int_{\Omega \cap\left\{\left|u_k\right|<1\right\}}\left|u_k\right|^{1-\gamma} d x\right) \leq C\left(\left\|u_k\right\|+\left\|u_k\right\|^{1-\gamma}\right).  
    \end{array}
\end{equation}
where $C$ is a positive constant independent of $k$. Then inserting \cref{step2} into \cref{step1}, and as $r+1>p$, we get for some positive constant $C_1$ independent of $k$,
\begin{equation}{\label{step3}}
I_{\lambda, \epsilon}(u_k)-\frac{1}{r+1} I_{\lambda, \epsilon}^{\prime}(u_k) u_k \geq C_1\left\|u_k\right\|^p-C\left(\left\|u_k\right\|+\left\|u_k\right\|^{1-\gamma}\right) .
\end{equation}
As $I^\prime_{\lambda,\epsilon}(u_k)\rightarrow 0$, therefore for $\eta>0$, there exists $k$ large enough such that
\begin{equation*}
    \left|I_{\lambda,\epsilon}^\prime(u_k)(v)\right|\leq \eta \|v\|,
\end{equation*}
for all $v\in W^{1,p}_0(\Omega)$. Now choosing $\eta=1$ and $v=u_k$, and using \cref{psc}, it holds for large enough $k$,
\begin{equation}{\label{step4}}
\left|I_{\lambda, \epsilon}(u_k)-\frac{1}{r+1} I_{\lambda, \epsilon}^{\prime}(u_k) u_k\right| \leq \left|I_{\lambda, \epsilon}(u_k)\right|+\left|\frac{1}{r+1} I_{\lambda, \epsilon}^{\prime}(u_k) u_k\right|\leq c+1+C\left\|u_k\right\|. 
\end{equation}
From \cref{step3,step4}, using Young's inequality, our claim follows as $p>1$ and $0<\gamma<1$.\smallskip\\ Since $\left\{u_k\right\}$ is bounded in the reflexive Banach space $W^{1,p}_0(\Omega)$, there exists $u_0 \in W^{1,p}_0(\Omega)$ such that up to a subsequence, $u_k \rightharpoonup u_0$ weakly in $W^{1,p}_0(\Omega)$,  $u_k \rightarrow u_0$ in $L^q(\Omega)$ for all $q\in[1,p^*)$ and $u_k\rightarrow u_0$ a.e. in $\Omega$.
We now show $u_k \rightarrow u_0$ strongly in $ W^{1,p}_0(\Omega)$ as $k \rightarrow \infty$. Denoting
\begin{equation*}
\mathcal{E}(\phi, \psi):=\int_{\mathbb{R}^{n}}\int_{\mathbb{R}^{n}}\frac{|\phi(x)-\phi(y)|^{p-2}(\phi(x)-\phi(y))(\psi(x)-\psi(y))}{|x-y|^{n+s p}} d x d y ,
\end{equation*}
we have by \cref{psc} that
\begin{equation*}
\lim _{k \rightarrow \infty}\left(\int_\Omega|\nabla u_k|^{p-2}\nabla u_k\cdot \nabla u_0 \,d x +\mathcal{E}(u_k, u_0)-\lambda \int_{\Omega}\left(u_k^{+}+\epsilon\right)^{-\gamma} u_0\, d x-\int_{\Omega}\left(u_k^{+}\right)^r u_0\, d x\right)=0    
\end{equation*}
and
\begin{equation*}
\lim _{k \rightarrow \infty}\left(\int_\Omega|\nabla u_k|^{p-2}\nabla u_k\cdot \nabla u_k \,d x +\mathcal{E}(u_k, u_k)-\lambda \int_{\Omega}\left(u_k^{+}+\epsilon\right)^{-\gamma} u_k \,d x-\int_{\Omega}\left(u_k^{+}\right)^r u_k\, d x\right)=0 .   
\end{equation*}
Now setting $U_k(x, y)=u_k(x)-u_k(y)$, $U_0(x, y)=u_0(x)-u_0(y)$ and subtracting the above two, one obtains
\begin{equation}{\label{step5}}
    \begin{array}{c}
         \lim _{k \rightarrow \infty} \int_\Omega (|\nabla u_k|^{p-2}\nabla u_k-|\nabla u_0|^{p-2}\nabla u_0)\cdot(\nabla u_k-\nabla u_0) d x \\+ \lim _{k \rightarrow \infty}  \int_{\mathbb{R}^{n}}\int_{\mathbb{R}^{n}} \frac{\left(\left|U_k(x, y)\right|^{p-2} U_k(x, y)-\left|U_0(x, y)\right|^{p-2} U_0(x, y)\right)\left(U_k(x, y)-U_0(x, y)\right)}{|x-y|^{n+s p}} d x d y \smallskip\\
=  \lim _{k \rightarrow \infty}\left(\lambda \int_{\Omega}\left(u_k^{+}+\epsilon\right)^{-\gamma} u_k\, d x+\int_{\Omega}\left(u_k^{+}\right)^r u_k\, d x-\lambda \int_{\Omega}\left(u_k^{+}+\epsilon\right)^{-\gamma} u_0 \,d x-\int_{\Omega}\left(u_k^{+}\right)^r u_0\, d x\right)\smallskip\\ -\lim _{k \rightarrow \infty} \int_\Omega (|\nabla u_0|^{p-2}\nabla u_0\cdot\nabla u_k-|\nabla u_0|^p) d x -\lim _{k \rightarrow \infty}(\mathcal{E}(u_0, u_k)-\mathcal{E}\left(u_0, u_0\right)) .
    \end{array}
\end{equation}
As $\{u_k\}$ is uniformly bounded is $W^{s,p}_0(\Omega)$ too, and $u_k\to u_0$ a.e. in $\Omega$, therefore $u_k \rightharpoonup u_0$ in $W^{s,p}_0(\Omega)$ and we have
\begin{equation*}
\frac{u_k(x)-u_k(y)}{|x-y|^{\frac{n+s p}{p}}} \rightharpoonup \frac{u_0(x)-u_0(y)}{|x-y|^{\frac{n+s p}{p}}}    
\end{equation*}
weakly in $L^p(\mathbb{R}^{2 n})$. Now since
\begin{equation*}
    \frac{\left|u_0(x)-u_0(y)\right|^{p-2}\left(u_0(x)-u_0(y)\right)}{|x-y|^{\frac{n+s p}{p^{\prime}}}} \in L^{p^{\prime}}(\mathbb{R}^{2 n}),
\end{equation*}
one gets
\begin{equation}{\label{term1}}
\lim _{k \rightarrow \infty}\left(\mathcal{E}\left(u_0, u_k\right)-\mathcal{E}\left(u_0, u_0\right)\right)=0 .
\end{equation} 
Similarly, $u_k\rightharpoonup u_0$ in $W^{1,p}_0(\Omega)$ implies $\nabla u_k\rightharpoonup\nabla u_0$ in $L^p(\Omega)$ and as $|\nabla u_0|^{p-2}\nabla u_0\in L^{p^\prime}(\Omega)$, therefore one can obtain
\begin{equation}{\label{term2}}
    \lim _{k \rightarrow \infty} \int_\Omega (|\nabla u_0|^{p-2}\nabla u_0\cdot\nabla u_k-|\nabla u_0|^p) d x =0.
\end{equation}
Also, observe that $|(u_k^{+}+\epsilon)^{-\gamma} u_0| \leq \epsilon^{-\gamma} u_0$ and $u_0\in L^1(\Omega)%\int_{\Omega}|%\epsilon^{-\gamma} 
%u_0| d x<+\infty
$. Therefore by Lebesgue Dominated convergence theorem, it holds
\begin{equation}{\label{term3}}
\lim _{k \rightarrow \infty} \int_{\Omega}\left(u_k^{+}+\epsilon\right)^{-\gamma} u_0 \,d x=\int_{\Omega}\left(u_0^{+}+\epsilon\right)^{-\gamma} u_0\, d x .
\end{equation}
Since $u_k \rightarrow u_0$ pointwise a.e. in $\Omega$ and for any measurable subset $E$ of $\Omega$ we have
\begin{equation*}    
\int_E\left|\left(u_k^{+}+\epsilon\right)^{-\gamma} u_k\right| d x \\
\leq  \int_E \epsilon^{-\gamma} |u_k| d x \leq\epsilon^{-\gamma}\left\|u_k\right\|_{L^{p^*}(\Omega)}|E|^{\frac{p^*-1}{p^*}} \leq C(\epsilon)|E|^{\frac{p^*-1}{p^*}},
\end{equation*}
so by Vitali convergence theorem, it follows that
\begin{equation}{\label{term4}}
\lim _{k \rightarrow \infty} \lambda \int_{\Omega}\left(u_k^{+}+\epsilon\right)^{-\gamma} u_k \,d x=\lambda \int_{\Omega}\left(u_0^{+}+c\right)^{-\gamma} u_0\, d x .
\end{equation}
Similarly, since $rp^{*\prime}<p^*$, we have\begin{equation*}
\int_E\left|\left(u_k^{+}\right)^r u_0\right| d x \leq\left\|u_0\right\|_{L^{p^*}(\Omega)}\left(\int_E\left(u_k^{+}\right)^{r p^{* \prime}} d x\right)^{\frac{1}{p^{* \prime}}} \leq C_3|E|^\alpha
\end{equation*}
and
\begin{equation*}
\int_E\left|\left(u_k^{+}\right)^r u_k\right| d x \leq\left\|u_k\right\|_{L^{p^*}(\Omega)}\left(\int_E\left(u_k^{+}\right)^{r p^{* \prime}} d x\right)^{\frac{1}{p^{* \prime}}} \leq C_4|E|^\beta,\end{equation*}
for some positive constants $C_3, C_4, \alpha$ and $\beta$ independent of $k$. Therefore by Vitali convergence theorem,
\begin{equation}{\label{term5}}
    \lim _{k \rightarrow \infty} \int_{\Omega}\left(u_k^{+}\right)^r u_0\, d x=\int_{\Omega}\left(u_0^{+}\right)^r u_0 \,d x,
\end{equation}
and
\begin{equation}{\label{term6}}
\lim _{k \rightarrow \infty} \int_{\Omega}\left(u_k^{+}\right)^r u_k\, d x=\int_{\Omega}\left(u_0^{+}\right)^r u_0\, d x .
\end{equation}
Employing \cref{term1,term2,term3,term4,term5,term6} in \cref{step5}, it now follows
\begin{equation}{\label{convergence}}
    \begin{array}{c}
               \lim _{k \rightarrow \infty} \int_\Omega (|\nabla u_k|^{p-2}\nabla u_k-|\nabla u_0|^{p-2}\nabla u_0)\cdot(\nabla u_k-\nabla u_0) d x \smallskip\\+ \lim _{k \rightarrow \infty} \int_{\mathbb{R}^{n}}\int_{\mathbb{R}^{n}} \frac{\left(\left|U_k(x, y)\right|^{p-2} U_k(x, y)-\left|U_0(x, y)\right|^{p-2} U_0(x, y)\right)\left(U_k(x, y)-U_0(x, y)\right)}{|x-y|^{n+s p}} d x d y =0.
    \end{array}
\end{equation}
\cref{p case} implies that both the terms on the above are nonnegative and hence individually go to $0$. Moreover, by Hölder's inequality and item $(i)$ of \cref{algebraic}, one gets that
\begin{equation}{\label{step6}}
\begin{array}{c}
   %\lim _{k \rightarrow \infty}
  \int_{\mathbb{R}^{n}}\int_{\mathbb{R}^{n}} \frac{\left(\left|U_k(x, y)\right|^{p-2} U_k(x, y)-\left|U_0(x, y)\right|^{p-2} U_0(x, y)\right)\left(U_k(x, y)-U_0(x, y)\right)}{|x-y|^{n+s p}} d x d y \smallskip\\
\geq\left(\left\|u_k\right\|_{W^{s,p}_0(\Omega)}^{p-1}-\left\|u_0\right\|_{W^{s,p}_0(\Omega)}^{p-1}\right)\left(\left\|u_k\right\|_{W^{s,p}_0(\Omega)}-\left\|u_0\right\|_{W^{s,p}_0(\Omega)}\right)\smallskip\\\geq C(p)\left(\|u_k\|^{p/2}_{W^{s,p}_0(\Omega)}-\|u_0\|^{p/2}_{W^{s,p}_0(\Omega)}\right)^2\geq 0,  
\end{array}
\end{equation}
and similarly
\begin{equation*}{\label{step7}}
    \int_\Omega (|\nabla u_k|^{p-2}\nabla u_k-|\nabla u_0|^{p-2}\nabla u_0)\cdot(\nabla u_k-\nabla u_0) d x \geq C(p)\bigg(\left(\int_\Omega |\nabla u_k|^p d x \right)^{1/2}-\left(\int_\Omega |\nabla u_0|^p d x \right)^{1/2}\bigg)^2\geq0.
\end{equation*}
Clubbing this with \cref{step6} and using \cref{convergence}, we obtain $\|u_k\|_{W^{1,p}_0(\Omega)}\rightarrow\|u_0\|_{W^{1,p}_0(\Omega)}$. This along with the weak convergence $u_k\rightharpoonup u_0$ in the uniformly convex Banach space $W^{1,p}_0(\Omega)$ gives us $u_k\rightarrow u_0$ strongly as $k\to \infty$.%, in $W^{1,p}_0(\Omega).$
\end{proof}
Next, we show $I_{\lambda,\epsilon}$ satisfies the Mountain Pass Geometry.
\begin{lemma}{\label{mountainpass}} There exist $R>0, \rho>0$ and $\Lambda>0$ depending on $R$ such that $\inf _{\|v\| \leq R} I_{\lambda, \epsilon}(v)<0$ and $\inf _{\|v\|=R} I_{\lambda, \epsilon}(v) \geq \rho$, for  every $\epsilon>0$ and $\lambda \in(0, \Lambda)$. Moreover there exists $T>R$ such that $I_{\lambda, \epsilon}(T e_1)<-1$, for each $\epsilon>0$ and $\lambda \leq \Lambda%\in(0, \Lambda)
$.
\end{lemma}
\begin{proof}Fixing $l=|\Omega|^{\frac{1}{\left(\frac{p^*}{r+1}\right)^{\prime}}}$ and using Hölder's inequality along with \cref{Sobolev embedding} we get that
\begin{equation}{\label{res1}}
\int_{\Omega}\left(v^{+}\right)^{r+1} d x\leq \int_\Omega |v|^{r+1} d x \leq\left(\int_{\Omega}|v|^{p^*}\right)^{\frac{r+1}{p^*}}|\Omega|^{1 /\left(\frac{p^*}{r+1}\right)^{\prime}} \leq C l\|v\|^{r+1},
\end{equation}
for some positive constant $C$ which does not depend on $v$. On the other hand, note that
\begin{equation*}
\lim _{t \rightarrow 0} \frac{I_{\lambda, \epsilon}(t e_1)}{t}=-\lambda \int_{\Omega} \epsilon^{-\gamma} e_1 d x<0,
\end{equation*}
and hence we can choose $k \in(0,1)$ sufficiently small and set $R:=k\left(\frac{r+1}{p C l}\right)^{\frac{1}{r+1-p}}$ such that \begin{equation*}
    \inf _{\|v\| \leq R} I_{\lambda, \epsilon}(v)<0.
\end{equation*} Moreover, since $R<\left(\frac{r+1}{p C l}\right)^{\frac{1}{r+1-p}}$, using \cref{res1} we obtain for $\|v\|=R$,
\begin{equation}{\label{res2}}
I_{0, \epsilon}(v) \geq \frac{R^p}{p}-\frac{C l R^{r+1}}{r+1}:=2 \rho \text { (say) }>0 .\end{equation}
Define now
\begin{equation*}
\Lambda:=\frac{\rho}{\underset{\|v\|=R}{\sup}\left(\frac{1}{1-\gamma} \int_{\Omega}|v|^{1-\gamma} d x\right)},
\end{equation*}
which is a positive constant and since $\rho, R$ depends on $k, r, p,|\Omega|, C$ so does $\Lambda$. Therefore \cref{fracineq} gives
\begin{equation*}
    \begin{array}{rcl}
         I_{\lambda, \epsilon}(v) & \geq &\frac{1}{p}\int_\Omega |\nabla v|^p d x +\frac{1}{p} \int_{\mathbb{R}^{n}}\int_{\mathbb{R}^{n}} \frac{|v(x)-v(y)|^p}{|x-y|^{n+s p}} d x d y-\frac{1}{r+1} \int_{\Omega}\left(v^{+}\right)^{r+1} d x-\frac{\lambda}{1-\gamma} \int_{\Omega}\left(v^{+}\right)^{1-\gamma} d x \smallskip\\
& =&I_{0, \epsilon}(v)-\frac{\lambda}{1-\gamma} \int_{\Omega}\left(v^{+}\right)^{1-\gamma} d x .
    \end{array}
\end{equation*}
Hence by \cref{res2},
\begin{equation*}
    \begin{array}{rcl}
         \inf _{\|v\|=R} I_{\lambda, \epsilon}(v) & \geq &\inf _{\|v\|=R} I_{0, \epsilon}(v)-\lambda \sup _{\|v\|=R}\left(\frac{1}{1-\gamma} \int_{\Omega}|v|^{1-\gamma} d x\right) \geq 2 \rho-\lambda \sup _{\|v\|=R}\left(\frac{1}{1-\gamma} \int_{\Omega}|v|^{1-\gamma} d x\right) \geq \rho,
    \end{array}
\end{equation*}
if $\lambda \in(0, \Lambda)$. Finally, it is easy to observe that $I_{0, \epsilon}\left(t e_1\right) \rightarrow-\infty$ as $t \rightarrow+\infty$, and hence we can choose $T>R$ such that $I_{0, \epsilon}\left(T e_1\right)<-1$. Hence
\begin{equation*}
I_{\lambda, \epsilon}(T e_1) \leq I_{0, \epsilon}(T e_1)<-1,
\end{equation*}
which completes the proof.
\end{proof}
\begin{remark}
As a consequence of \cref{mountainpass}, we have
\begin{equation*}
    \inf _{\|v\|=R} I_{\lambda, \epsilon}(v) \geq \rho \max \left\{I_{\lambda, \epsilon}(T e_1), I_{\lambda, \epsilon}(0)\right\}=0 .
\end{equation*}
\end{remark}
\begin{remark}{\label{appmountainpass}}
    Using \cref{ps} and \cref{mountainpass}, by Mountain Pass Theorem, we get for every $\lambda \in(0, \Lambda)$, there exists a $w_\epsilon \in W^{1,p}_0(\Omega)$ such that $I_{\lambda, \epsilon}^{\prime}(w_\epsilon)=0$ and
\begin{equation*}
I_{\lambda, \epsilon}(w_\epsilon)=\inf _{\gamma \in \Gamma} \max _{t \in[0,1]} I_{\lambda, \epsilon}(\gamma(t)) \geq \rho>0,
\end{equation*}
where $\Gamma=\left\{\gamma \in C([0,1], W^{1,p}_0(\Omega)): \gamma(0)=0, \gamma(1)=T e_1\right\}$. Using \cref{fracineq}, \cref{step2} together with Vitali convergence theorem, if $u_k \rightharpoonup u_0$ weakly in $W^{1,p}_0(\Omega)$, then we have
\begin{equation*}
    \lim _{k \rightarrow \infty} \int_{\Omega} \frac{\left(u_k+\epsilon\right)^{1-\gamma}-\epsilon^{1-\gamma}}{1-\gamma} d x=\int_{\Omega} \frac{\left(u_0+\epsilon\right)^{1-\gamma}-\epsilon^{1-\gamma}}{1-\gamma} d x .
\end{equation*}
This along with \cref{term6} and the fact that the norm function is weakly lower semicontinuous, we have $I_{\lambda, \epsilon}$ is also so. Moreover, from \cref{mountainpass}, as for every $\epsilon>0$ and $\lambda \in(0, \Lambda)$ we have $\inf _{\|v\| \leq R} I_{\lambda, \epsilon}(v)<0$, so there exists nonzero $v_\epsilon \in W^{1,p}_0(\Omega) $ such that $\left\|v_\epsilon\right\| \leq R$ and
\begin{equation}{\label{different}}
\inf _{\|v\| \leq R} I_{\lambda, \epsilon}(v)=I_{\lambda, \epsilon}(v_\epsilon)<0<\rho \leq I_{\lambda, \epsilon}(w_\epsilon) .    
\end{equation}
Therefore, $w_\epsilon$ and $v_\epsilon$ are two different non trivial critical points of $I_{\lambda, \epsilon}$, provided $\lambda \in$ $(0, \Lambda)$.
\end{remark}
\begin{remark}{\label{nonnegative}}Testing \cref{approximated}, with $\min \left\{w_\epsilon, 0\right\}$ and $\min \left\{v_\epsilon, 0\right\}$,  and noting that $\lambda(u^++\ep)^{-\gamma}+(u^+)^r$ is nonnegative in $\Omega$, one can use the technique of [\citealp{garain}, Lemma 3.1], to get the critical points $w_\epsilon$ and $v_\epsilon$ of $I_{\lambda, \epsilon}$ are nonnegative.
\end{remark}
Our next lemma states that the critical points are uniformly bounded in $W^{1,p}_0(\Omega)$.
\begin{lemma}{\label{uniform}}There exists $M>0$, constant(independent of $\epsilon$) such that $\left\|u_\epsilon\right\|_{W^{1,p}_0(\Omega)} \leq M$ where $u_\epsilon=w_\epsilon$ or $v_\epsilon$.
\end{lemma}
\begin{proof}Clearly, the result  is trivial if $u_{\epsilon}=v_\epsilon$, so we deal with the case $u_{\epsilon}=w_\epsilon$. We define \begin{equation*}
    A=\max _{t \in[0,1]} I_{0, \epsilon}(t T e_1)
\end{equation*}and use \cref{different} to get
\begin{equation*}
A \geq \max _{t \in[0,1]} I_{\lambda, \epsilon}(t T e_1) \geq \inf _{\gamma \in \Gamma} \max _{t \in[0,1]} I_{\lambda, \epsilon}(\gamma(t))=I_{\lambda, \epsilon}(w_\epsilon) \geq \rho>0>I_{\lambda, \epsilon}(v_\epsilon) .
\end{equation*}
Therefore
\begin{equation}{\label{bdd1}}
    \frac{1}{p}\int_\Omega |\nabla w_\epsilon|^p d x +\frac{1}{p} \int_{\mathbb{R}^{n}}\int_{\mathbb{R}^{n}}  \frac{|w_\epsilon(x)-w_\epsilon(y)|^p}{|x-y|^{n+s p}} d x d y-\lambda \int_{\Omega} \frac{\left(w_\epsilon+\epsilon\right)^{1-\gamma}-\epsilon^{1-\gamma}}{1-\gamma} d x-\frac{1}{r+1} \int_{\Omega} w_\epsilon^{r+1} d x \leq A .
\end{equation}
We choose $\phi=-\frac{w_{\epsilon}}{r+1}$ as a test function and note that $w_\epsilon$ is a weak solution of \cref{approximated}. Therefore we have the following
\begin{equation}{\label{bdd2}}
    -\frac{1}{r+1}\left(\int_\Omega |\nabla w_\epsilon|^p d x + \int_{\mathbb{R}^{n}}\int_{\mathbb{R}^{n}}  \frac{|w_\epsilon(x)-w_\epsilon(y)|^p}{|x-y|^{n+s p}} d x d y\right)+\frac{\lambda}{r+1} \int_{\Omega} \frac{w_\epsilon}{\left(w_\epsilon+\epsilon\right)^{\gamma}} d x+\frac{1}{r+1} \int_{\Omega} w_\epsilon^{r+1} d x=0 .
\end{equation}
Adding \cref{bdd1,bdd2} we get,% using \cref{fracineq,step2},
\begin{equation*}
\begin{array}{l}
\left(\frac{1}{p}-\frac{1}{r+1}\right)\left(\int_\Omega |\nabla w_\epsilon|^p d x +\frac{1}{p}\int_{\mathbb{R}^{n}}\int_{\mathbb{R}^{n}} \frac{|w_\epsilon(x)-w_\epsilon(y)|^p}{|x-y|^{n+s p}} d x d y\right)\smallskip\\
\leq  \lambda \int_{\Omega} \frac{\left(w_\epsilon+\epsilon\right)^{1-\gamma}-\epsilon^{1-\gamma}}{1-\gamma} d x-\frac{\lambda}{r+1} \int_{\Omega} \frac{w_\epsilon}{\left(w_\epsilon+\epsilon\right)^{\gamma}} d x+A \smallskip\\
\leq  \lambda \int_{\Omega} \frac{\left(w_\epsilon+\epsilon\right)^{1-\gamma}-\epsilon^{1-\gamma}}{1-\gamma} d x+A \leq C\left\|w_\epsilon\right\|^{1-\gamma}+A,
\end{array}
\end{equation*}
for some positive constant $C$ independent of $\epsilon$. Note that we have deduced the last inequality by using \cref{fracineq}, Hölder's inequality and \cref{Sobolev embedding}. Now since $r+1>p$ and $0<\gamma<1$, we conclude by using Young's inequality that the sequence $\left\{w_\epsilon\right\}$ is uniformly bounded in $W^{1,p}_0(\Omega)$ with respect to $\epsilon$.
\end{proof}
\begin{remark}{\label{u_0}}In view of \cref{nonnegative,uniform}, we can say that up to a subsequence $w_\epsilon \rightharpoonup w_0$ and $v_\epsilon \rightharpoonup v_0$ weakly in $W^{1,p}_0(\Omega)$ as $\epsilon \rightarrow 0$, for some nonnegative $w_0, v_0 \in W^{1,p}_0(\Omega)$. For convenience, we denote by $u_0$ either $w_0$ or $v_0$.
\end{remark}
We now give the result regarding the convergence of gradients of $u_\varepsilon$ to the gradient of $u$ a.e. in $\Omega$.
\begin{lemma}{\label{gradconv}}
Suppose $u_0$ be as in \cref{u_0}. Further assume that for each $\omega\subset\subset \Omega$, there exists a constant $c>0$, depending on $\omega$ such that for all $\varepsilon$, $u_\varepsilon\geq c(\omega)$ in $\omega$. Then up to a subsequence, $\nabla u_\varepsilon\rightarrow\nabla u$ pointwise almost everywhere in $\Omega$.
\end{lemma}
\begin{proof}
   Let us take a compact $K \subset \Omega$ and consider a function $\phi_K \in C_c^\infty(\Omega)$ such that $\operatorname{supp} \phi_K=\omega$, $0 \leq \phi_K \leq 1$ in $\Omega$ and $\phi_K \equiv 1$ in $K$. Now for $\mu>0$, we define the truncated function $T_\mu: \mathbb{R} \rightarrow \mathbb{R}$ by
\begin{equation*}
T_\mu(t)= \begin{cases}t, & \text { if }|t| \leq \mu, \\ \mu \frac{t}{|t|}, & \text { if }|t|>\mu .\end{cases}
\end{equation*}
Now choose $\psi_\epsilon=\phi_K T_\mu\left(\left(u_\epsilon-u_0\right)\right) \in W_0^{1, p}(\Omega)$ as a test function in \cref{approximated}, to get
\begin{equation*}
    I+J=R+S,
\end{equation*}
where
\begin{equation*}
    \begin{array}{c}
I=\int_{\Omega}\left|\nabla u_\epsilon\right|^{p-2} \nabla u_\epsilon \cdot \nabla \psi_\epsilon\, d x, \quad J=\int_{\mathbb{R}^n} \int_{\mathbb{R}^n} \frac{|u_\epsilon(x)-u_\epsilon(y)|^{p-2}(u_\epsilon(x)-u_\epsilon(y))(\psi_\epsilon(x)-\psi_\epsilon(y)) }{|x-y|^{n+ps}}d x d y  \smallskip\\
R=\int_{\Omega} \frac{\lambda \psi_\epsilon}{\left(u_\epsilon+\epsilon\right)^\gamma}  d x \quad \text { and } \quad  S=\int_\Omega (u_\epsilon)^r\psi_\epsilon \, dx .
    \end{array}
\end{equation*}
As $u_\epsilon$ is uniformly bounded in $W^{1,p}_0(\Omega)$, we estimate $S$ by using H\"older and Sobolev inequality as
\begin{equation*}
    S\leq \mu\int_\Omega(u_\epsilon)^rdx\leq \mu\left(\int_\Omega u_\epsilon \,dx\right)^r\leq C\mu(\|u_\epsilon\|_{L^p(\Omega)})^r \leq C\mu (\|u_\epsilon\|_{W^{1,p}_0(\Omega)})^r\leq C\mu,
\end{equation*}
for $p-1<r<1$ (the case when $1<p<2$). Now for $1\leq r< p^*-1<p^*$, one gets
\begin{equation*}
    S\leq \mu\int_\Omega(u_\epsilon)^rdx\leq C\mu\left(\int_\Omega(u_\epsilon)^{p^*}dx\right)^{r/p^*} \leq C\mu (\|u_\epsilon\|_{W^{1,p}_0(\Omega)})^r\leq C\mu.
\end{equation*}
Denoting by $\mathcal{A}\eta(x,y)=|\eta(x)-\eta(y)|^{p-2}(\eta(x)-\eta(y))$ and $d\nu=\frac{1}{|x-y|^{n+ps}} d x dy$, we can write $J$ as:
\begin{equation*}
    \begin{array}{rcl}
         J &= & \int_{\mathbb{R}^n} \int_{\mathbb{R}^n} \mathcal{A} u_\epsilon(x, y)(\psi_\epsilon(x)-\psi_\epsilon(y)) d \nu\smallskip \\
&= & \int_{\mathbb{R}^n} \int_{\mathbb{R}^n} \phi_K(x)(\mathcal{A} u_\epsilon(x, y)-\mathcal{A} u_0(x, y))(T_\mu((u_\epsilon-u_0)(x))-T_\mu((u_\epsilon-u_0)(y))) d \nu \smallskip\\
&& +\int_{\mathbb{R}^n} \int_{\mathbb{R}^n} T_\mu((u_\epsilon-u_0)(y)) \mathcal{A} u_\epsilon(x, y)(\phi_K(x)-\phi_K(y)) d \nu \smallskip\\
&& +\int_{\mathbb{R}^n} \int_{\mathbb{R}^n} T_\mu((u_\epsilon-u_0)(y)) \mathcal{A} u_0(x, y)(\phi_K(y)-\phi_K(x)) d \nu \smallskip\\
&& +\int_{\mathbb{R}^n} \int_{\mathbb{R}^n} \mathcal{A} u_0(x, y)(\phi_K(x) T_\mu((u_\epsilon-u_0)(x))-\phi_K(y) T_\mu((u_\epsilon-u_0)(y))) d \nu\smallskip \\
&&:= J_1+J_2+J_3+J_4 .
    \end{array}
\end{equation*}
We show
\begin{equation*}
    J_1=\int_{\mathbb{R}^n} \int_{\mathbb{R}^n} \phi_K(x)(\mathcal{A} u_\epsilon(x, y)-\mathcal{A} u_0(x, y))(T_\mu((u_\epsilon-u_0)(x))-T_\mu((u_\epsilon-u_0)(y))) d \nu\geq 0.
\end{equation*}
To this end, it is enough to prove that the integrand is nonnegative. We observe that
$
\mathbb{R}^n\times \mathbb{R}^n=\cup_{i=1}^4 S_i,$
where
\begin{equation*}
    \begin{array}{c}
         S_1=\left\{(x, y) \in \mathbb{R}^n \times \mathbb{R}^n:|(u_\epsilon-u_0)(x)| \leq \mu,|(u_\epsilon-u_0)(y)| \leq \mu\right\}, \smallskip\\
S_2=\left\{(x, y) \in \mathbb{R}^n \times \mathbb{R}^n:|(u_\epsilon-u_0)(x)| \leq \mu<|(u_\epsilon-u_0)(y)| \right\}, \smallskip\\
S_3=\left\{(x, y) \in \mathbb{R}^n \times \mathbb{R}^n:|(u_\epsilon-u_0)(y)| \leq \mu<|(u_\epsilon-u_0)(x)| \right\}    \end{array}
\end{equation*}
and
\begin{equation*}
    S_4=\left\{(x, y) \in \mathbb{R}^n \times \mathbb{R}^n:|(u_\epsilon-u_0)(x)| >\mu,|(u_\epsilon-u_0)(y)| >\mu\right\}.
\end{equation*}\smallskip
\textbf{Case $1$.} If $x, y \in S_1$, then, 
$
T_\mu((u_\epsilon-u_0)(x))=(u_\epsilon-u_0)(x)$ and $T_\mu((u_\epsilon-u_0)(y))=(u_\epsilon-u_0)(y) .
$
Therefore by \cref{p case}, it easily follows that $J_1\geq 0$.\smallskip\\
\textbf{Case $2$.} Let $x, y \in S_2$. Then, $|(u_\epsilon-u_0)(x)| \leq \mu<|(u_\epsilon-u_0)(y)|$ and we consider four cases.\\ Firstly when $u_0(x)\geq u_\ep(x)$ and $u_0(y)\geq u_\ep(y)$, we have $u_0(x)-u_\ep(x) \leq \mu<u_0(y)-u_\ep(y)$. Therefore,
\begin{equation*}
    T_\mu((u_\ep-u_0)(x))-T_\mu((u_\ep-u_0)(y))=(u_\ep-u_0)(x)+\mu \geq 0 .
\end{equation*}
Moreover, in this case, $u_\ep(x)-u_\ep(y)>u_0(x)-u_0(y)$. Hence by monotonicity of $t\to |t|^{p-2}t$ it holds
\begin{equation*}
\mathcal{A} u_\ep(x, y)-\mathcal{A} u_0(x, y) \geq 0     
\end{equation*}
and hence we get $J_1\geq 0$.\\
Second case occurs when $u_0(x)< u_\ep(x)$ and $u_0(y)\geq u_\ep(y)$, and in this case, $u_0(x)-u_\ep(x)<u_\ep(x)-u_0(x) \leq \mu<u_0(y)-u_\ep(y)$ and 
\begin{equation*}
    T_\mu((u_\ep-u_0)(x))-T_\mu((u_\ep-u_0)(y))=(u_\ep-u_0)(x)+\mu > 0 ,
\end{equation*}
and further, $u_\ep(x)-u_\ep(y)>u_0(x)-u_0(y)$ implies $\mathcal{A} u_\ep(x, y)-\mathcal{A} u_0(x, y) \geq 0$ giving $J_1\geq 0$.\\ For the third case $u_0(x)\geq u_\ep(x)$ and $u_0(y)< u_\ep(y)$, one has $u_\ep(x)-u_0(x) <u_0(x)-u_\ep(x)\leq \mu<u_\ep(y)-u_0(y)$.
Then 
\begin{equation*}
    T_\mu((u_\ep-u_0)(x))-T_\mu((u_\ep-u_0)(y))=(u_\ep-u_0)(x)-\mu \leq0 ,
\end{equation*}
and $u_\ep(x)-u_\ep(y)<u_0(x)-u_0(y)$ implies $\mathcal{A} u_\ep(x, y)-\mathcal{A} u_0(x, y) \leq 0$ will give $J_1\geq 0$.\\
Finally when $u_0(x)< u_\ep(x)$ and $u_0(y)< u_\ep(y)$, we have $%u_0(x)-u_\ep(x)<
u_\ep(x)-u_0(x) \leq \mu<u_\ep(y)-u_0(y)$. Therefore \begin{equation*}
    T_\mu((u_\ep-u_0)(x))-T_\mu((u_\ep-u_0)(y))=(u_\ep-u_0)(x)-\mu \leq0 ,
\end{equation*}
and $u_\ep(x)-u_\ep(y)\leq u_0(x)-u_0(y)$ implies $\mathcal{A} u_\ep(x, y)-\mathcal{A} u_0(x, y) \leq 0$. Hence we get $J_1\geq 0$.
\smallskip\\ Due to symmetry, the case for $x,y\in S_3$ will follow similarly like Case $2$. \smallskip\\
\textbf{Case 3.} Let $x, y \in S_4$. Then a case-by-case inspection clearly guarantees that \begin{equation*}
T_\mu((u_\ep-u_0)(x))-T_\mu((u_\ep-u_0)(y))=0,
\end{equation*}
if $u_0(x)<u_\ep(x)$, $u_0(y)< u_\ep(y)$ or if $u_0(x)\geq u_\ep(x)$, $u_0(y)\geq u_\ep(y)$. \\Further $u_0(x)\geq u_\ep(x)$, $u_0(y)< u_\ep(y)$ gives $T_\mu((u_\ep-u_0)(x))-T_\mu((u_\ep-u_0)(y))=-2\mu<0$; and $u_0(x)< u_\ep(x)$, $u_0(y)\geq u_\ep(y)$ gives $T_\mu((u_\ep-u_0)(x))-T_\mu((u_\ep-u_0)(y))=2\mu>0$. In each case, one can easily check the sign of $\mathcal{A} u_\ep(x, y)-\mathcal{A} u_0(x, y) $ and find that $J_1\geq 0$.\smallskip\\
Combining all the three cases, we conclude
$
    J_1 \geq 0.
$
The rest of the proof follows from [\citealp{garain}, Theorem A.1].
\end{proof}
We now establish that $w_0 \neq v_0$ are weak solutions to \cref{prob}. 
\begin{lemma}{\label{weaksolutiontoprob}}$u_0 \in W^{1,p}_0(\Omega)$ is a weak solution to the problem \cref{prob}.
\end{lemma}
\begin{proof}
We first observe that for any $\epsilon \in(0,1)$ and $t \geq 0$,
\begin{equation*}
\frac{\lambda}{(t+\epsilon)^{\gamma}}+t^r \geq \frac{\lambda}{(t+1)^{\gamma}}+t^r \geq \min \left\{1, \frac{\lambda}{2^\gamma}\right\}:=B \text{ (say)} .    
\end{equation*}
Hence we can write at least for small $\epsilon$,
\begin{equation*}
-\Delta_p u_\epsilon+(-\Delta)_p^s u_\epsilon=\frac{\lambda}{\left(u_\epsilon+\epsilon\right)^{\gamma}}+u_\epsilon^r \geq B.    
\end{equation*}Now by [\citealp{garain}, Lemma 3.1], we get the existence of a unique $\zeta \in W^{1,p}_0(\Omega)\cap L^\infty(\Omega)$ satisfying
\begin{equation*}
\begin{array}{c}
-\Delta_p\zeta+(-\Delta)_p^s \zeta=B  \text { in } \Omega, \smallskip\\\zeta>0 \text{ in } \Omega,\quad\quad \zeta=0  \text { in }\mathbb{R}^n \backslash \Omega;
\end{array}
\end{equation*}
such that for every $\omega\subset\subset\Omega$, $\exists\,c(\omega)>0$ satisfying $\zeta\geq c(\omega)>0$ in $\omega$.
Now, for every $\phi\in W^{1,p}_0(\Omega)$, it holds
\begin{equation}{\label{existence1}}
    \begin{array}{l}
       \quad\int_\Omega |\nabla u_\epsilon|^{p-2}\nabla u_\epsilon\cdot\nabla \phi  + \int_{\mathbb{R}^{n}}\int_{\mathbb{R}^{n}}  \frac{\left|u_\epsilon(x)-u_\epsilon(y)\right|^{p-2}\left(u_\epsilon(x)-u_\epsilon(y)\right)(\phi(x)-\phi(y))}{|x-y|^{n+s p}}  =\int_\Omega\left(
       \frac{\lambda}{(u_\epsilon+\epsilon)^\gamma} +u_\epsilon^r\right)\phi \,d x\smallskip\\
       \geq \int_\Omega B\phi \,dx =\int_\Omega |\nabla \zeta|^{p-2}\nabla \zeta\cdot\nabla \phi \,d x+ \int_{\mathbb{R}^{n}}\int_{\mathbb{R}^{n}} \frac{|\zeta(x)-\zeta(y)|^{p-2}(\zeta(x)-\zeta(y))(\phi(x)-\phi(y))}{|x-y|^{n+s p}} d x d y.
    \end{array}
\end{equation}
We now choose $\phi=(\zeta-u_\epsilon)^{+} \in W^{1,p}_0(\Omega)$ in \cref{existence1}, to get
\begin{equation*}
    \begin{array}{c}
       \int_\Omega (|\nabla \zeta|^{p-2}\nabla \zeta-|\nabla u_\epsilon|^{p-2}\nabla u_\epsilon)\cdot\nabla (\zeta-u_\epsilon)^{+} d x\smallskip\\ + \iint_{\mathbb{R}^{2 n}} \frac{(\left|\zeta(x)-\zeta(y)\right|^{p-2}\left(\zeta(x)-\zeta(y)\right)-\left|u_\epsilon(x)-u_\epsilon(y)\right|^{p-2}\left(u_\epsilon(x)-u_\epsilon(y)\right))((\zeta-u_\epsilon)^{+}(x)-(\zeta-u_\epsilon)^{+}(y))}{|x-y|^{n+s p}} d x d y \smallskip\\\leq 0.
       \end{array}
\end{equation*}
Following the same arguments as in the proof of [\citealp{FractionalEigenvalues}, Lemma 9], we have that the second nonlocal double integral in the above inequality is nonnegative. Hence it holds
\begin{equation*}
      \int_\Omega (|\nabla \zeta|^{p-2}\nabla \zeta-|\nabla u_\epsilon|^{p-2}\nabla u_\epsilon)\cdot\nabla (\zeta-u_\epsilon)^{+} d x\leq 0.
\end{equation*}
Then using \cref{p case}, one gets $u_\epsilon \geq \zeta$ in $\Omega$. Hence there exists $c(\omega)>0$ such that for small $\epsilon$, it holds 
\begin{equation*}
u_\epsilon \geq c(\omega)>0 \text { in } \omega\subset\subset\Omega \text {. }
\end{equation*}
This gives $u_0 \geq c(\omega)>0$ for every $\omega \subset\subset \Omega, u_0>0$ in $\Omega$ and
\begin{equation*}
0\leq\left|\frac{\lambda \phi}{\left(v_\epsilon+\epsilon\right)^{\gamma}}\right| \leq \lambda c^{-\gamma}\left\|\phi \right\|_{L^{\infty}(\Omega)}, \text { for every } \phi \in C_c^{\infty}(\Omega) .    
\end{equation*}
Hence by the dominated convergence theorem, we get for every $ \phi \in C_c^{\infty}(\Omega),$
\begin{equation}{\label{add1}}
    \lim _{\epsilon \rightarrow 0^+ }\int_{\Omega} \frac{\lambda}{\left(u_\epsilon+\epsilon\right)^{\gamma}} \phi \,d x=\int_{\Omega} \frac{\lambda}{u_0^{\gamma}} \phi \,d x .
\end{equation}
As $\{u_\epsilon\}$ bounded in $W^{1,p}_0(\Omega)$ implies $\{u_\epsilon^r\}$ is bounded in $L^{(r+1)^\prime}(\Omega)$ and $u_\ep\to u_0$ a.e. in $\Omega$, hence up to a subsequence  $u^r_\epsilon\rightharpoonup u_0^r$ in $L^{(r+1)^\prime}(\Omega)$ and it holds
\begin{equation}{\label{add2}}
\lim _{\epsilon \rightarrow 0^{+}} \int_{\Omega} u_\epsilon^r \phi \,d x=\int_{\Omega} u_0^r \phi \,d x, \forall \phi \in C_c^{\infty}(\Omega) .    
\end{equation}
Now $\{u_\epsilon\}$ is bounded in $W^{1,p}_0(\Omega)$, so by \cref{embedding2} we get \begin{equation*}
\frac{\left|u_\epsilon(x)-u_\epsilon(y)\right|^{p-2}\left(u_\epsilon(x)-u_\epsilon(y)\right)}{|x-y|^{\frac{n+s p}{p^\prime}}} \rightharpoonup \frac{\left|u_0(x)-u_0(y)\right|^{p-2}\left(u_0(x)-u_0(y)\right)}{|x-y|^{\frac{n+s p}{p^\prime}}}\end{equation*}
in $L^{p^\prime}(\mathbb{R}^{2n})$ and hence
\begin{equation}{\label{add3}}
\begin{array}{l}
\quad \lim _{\epsilon \rightarrow 0^{+}} \int_{\mathbb{R}^{n}}\int_{\mathbb{R}^{n}}  \frac{\left|u_\epsilon(x)-u_\epsilon(y)\right|^{p-2}\left(u_\epsilon(x)-u_\epsilon(y)\right)(\phi(x)-\phi(y))}{|x-y|^{n+s p}} d x d y \smallskip\\
= \int_{\mathbb{R}^{n}}\int_{\mathbb{R}^{n}}  \frac{\left|u_0(x)-u_0(y)\right|^{p-2}\left(u_0(x)-u_0(y)\right)(\phi(x)-\phi(y))}{|x-y|^{n+s p}} d x d y, \quad \forall \phi \in C_c^{\infty}(\Omega) .    
\end{array}
\end{equation}
Further, using \cref{gradconv}, %and the boundedness of $\{u_\ep\}$ in $W^{1,p}_0(\Omega)$, 
one can get $|\nabla u_\epsilon|^{p-2}\nabla u_\epsilon \rightharpoonup |\nabla u_0|^{p-2}\nabla u_0$ in $L^{p^\prime}(\Omega)$ and as $\nabla \phi\in L^p(\Omega)$, so it holds
\begin{equation}{\label{add4}}
 \lim _{\epsilon \rightarrow 0^{+}}\int_\Omega |\nabla u_\epsilon|^{p-2}\nabla u_\epsilon\cdot\nabla \phi \, dx =\int_\Omega |\nabla u_0|^{p-2}\nabla u_0\cdot\nabla \phi \, dx .
\end{equation}
Using \cref{add1,add2,add3,add4}, we conclude
\begin{equation*}
    \begin{array}{c}
\int_\Omega |\nabla u_0|^{p-2}\nabla u_0\cdot\nabla \phi \, dx +\iint_{\mathbb{R}^{2 n}} \frac{\left|u_0(x)-u_0(y)\right|^{p-2}\left(u_0(x)-u_0(y)\right)(\phi(x)-\phi(y))}{|x-y|^{n+s p}} d x d y=\int_{\Omega} \left(\frac{\lambda}{u_0^{\gamma}}+u_0^r\right)\phi \,d x,
    \end{array}
\end{equation*}
$\forall \phi \in C_c^{\infty}(\Omega)$. This completes the proof.
\end{proof}
\begin{remark}{\label{testfunc}}
Following the lines of [\citealp{garain}, Lemma 5.1], it can be shown that any function in $W^{1,p}_0(\Omega)$ can be chosen as a test function for \cref{prob}.
\end{remark}
\section{Proof of \cref{mainth1}}{\label{th1}} Using \cref{weaksolutiontoprob} we already have $w_0$ and $v_0$ are two positive weak solutions of \cref{prob} for $\lambda \in(0, \Lambda)$. It suffices to show that $w_0 \neq v_0$. Choosing $\phi=u_\epsilon \in W^{1,p}_0(\Omega)$ as a test function in \cref{approximated} we get
\begin{equation*}
    \begin{array}{l}
 \int_\Omega |\nabla u_\epsilon|^p d x+   \int_{\mathbb{R}^{n}}\int_{\mathbb{R}^{n}}  \frac{\left|u_\epsilon(x)-u_\epsilon(y)\right|^{p}}{|x-y|^{n+s p}} d x d y 
= \lambda \int_{\Omega} \frac{u_\epsilon}{\left(u_\epsilon+\epsilon\right)^{\gamma}} d x+\int_{\Omega} u_\epsilon^{r+1} d x .
    \end{array}
\end{equation*}
Since $r+1<p_s^*$, using the compact embedding of \cref{Sobolev embedding}, one obtains
\begin{equation}{\label{crucial}}
\lim _{\epsilon \rightarrow 0^{+}} \int_{\Omega}u_\epsilon^{r+1} d x=\int_{\Omega} u_0^{r+1} d x .    
\end{equation}
Moreover, since
%\begin{equation*}
$0\leq \frac{u_\epsilon}{\left(u_\epsilon+\epsilon\right)^{\gamma}} \leq u_\epsilon^{1-\gamma},    
$ %\end{equation*}
using \cref{step2} together with Vitali convergence theorem, it holds
\begin{equation*}
\lambda \lim _{\epsilon \rightarrow 0^{+}} \int_{\Omega} \frac{u_\epsilon}{\left(u_\epsilon+\epsilon\right)^{\gamma}} d x=\lambda \int_{\Omega} u_0^{1-\gamma} d x .
\end{equation*}
Hence we reach at
\begin{equation}{\label{laststep1}}
    \begin{array}{l}
 \lim _{\epsilon \rightarrow 0^{+}}\left(\int_\Omega |\nabla u_\epsilon|^p d x+  \int_{\mathbb{R}^{n}}\int_{\mathbb{R}^{n}}  \frac{\left|u_\epsilon(x)-u_\epsilon(y)\right|^{p}}{|x-y|^{n+s p}} d x d y \right)
= \lambda \int_{\Omega} u_0^{1-\gamma} d x +\int_{\Omega} u_0^{r+1} d x .
    \end{array}
\end{equation}
Using \cref{testfunc} we can choose $\phi=u_0$ as a test function in \cref{prob} to deduce that
\begin{equation}{\label{laststep2}}
    \begin{array}{l}
 \int_\Omega |\nabla u_0|^p d x+  \int_{\mathbb{R}^{n}}\int_{\mathbb{R}^{n}}  \frac{\left|u_0(x)-u_0(y)\right|^{p}}{|x-y|^{n+s p}} d x d y 
= \lambda \int_{\Omega} u_0^{1-\gamma} d x +\int_{\Omega} u_0^{r+1} d x .
    \end{array}
\end{equation}
Merging \cref{laststep1,laststep2}, we get
\begin{equation*}
    \begin{array}{l}
 \lim _{\epsilon \rightarrow 0^{+}}\left(\int_\Omega |\nabla u_\epsilon|^p d x+  \int_{\mathbb{R}^{n}}\int_{\mathbb{R}^{n}}  \frac{\left|u_\epsilon(x)-u_\epsilon(y)\right|^{p}}{|x-y|^{n+s p}} d x d y \right)
= \int_\Omega |\nabla u_0|^p d x+  \int_{\mathbb{R}^{n}}\int_{\mathbb{R}^{n}}  \frac{\left|u_0(x)-u_0(y)\right|^{p}}{|x-y|^{n+s p}} d x d y ,%\lambda \int_{\Omega} u_0^{1-\gamma} d x +\int_{\Omega} u_0^{r+1} d x ,
    \end{array}
\end{equation*}
which implies the strong convergence of $u_\epsilon$ to $u_0$ in $W^{1,p}_0(\Omega)$. Now again by \cref{step2} and Vitali convergence theorem, one can get
\begin{equation*}
    \lim _{\epsilon \rightarrow 0} \int_{\Omega}\left[\left(u_\epsilon+\epsilon\right)^{1-\gamma}-\epsilon^{1-\gamma}\right] d x=\int_{\Omega} u_0^{1-\gamma} d x,
\end{equation*}
which together with \cref{crucial} and the strong convergence of $u_\epsilon$ implies $\lim _{\epsilon \rightarrow 0} I_{\lambda, \epsilon}(u_\epsilon)=I_\lambda(u_0)$. Hence, from \cref{different} we deduce $w_0 \neq v_0$.
\section{Proof of \cref{mainth2}}{\label{th2}}
We first include the following lemma.
\begin{lemma}{\label{boundednessofsubsolution}}
    For each $\gamma>0$, there exists a positive constant $T$ such that every $z\in W^{1,2}_0(\Omega)$, $z>0$ satisfying 
    \begin{equation}{\label{bddine}}
        \int_\Omega \nabla z\cdot \nabla \phi \, d x+\int_{\mathbb{R}^n}\int_{\mathbb{R}^n}\frac{(z(x)-z(y))(\phi(x)-\phi(y))}{|x-y|^{n+2s}} d x d y\leq \int_\Omega \frac{\lambda}{z^\gamma}\phi\, dx,\quad \forall \phi\in W^{1,2}_0(\Omega),\, \phi>0,
    \end{equation}
    belongs to $L^\infty(\Omega)$ with $\|z\|_{L^\infty(\Omega)}\leq T\lambda^{\frac{1}{\gamma+1}}$ for all $\lambda>0$.
\end{lemma}
\begin{proof}
    If $\lambda=1$, then for $l\geq 1$ we choose $\phi=(z-l)^+$ as a test function to obtain
    \begin{equation*}
         \int_\Omega| \nabla (z-l)^+|^2 d x+\int_{\mathbb{R}^n}\int_{\mathbb{R}^n}\frac{(z(x)-z(y))((z-l)^+(x)-(z-l)^+(y))}{|x-y|^{n+2s}} d x d y\leq \int_{\{z>l\}} \frac{(z-l)^+}{z^\gamma}dx.
    \end{equation*}
Note that the second term in the left-hand-side is nonnegative and hence we get
 \begin{equation*}
         \int_\Omega| \nabla (z-l)^+|^2 d x%+\int_{\mathbb{R}^n}\int_{\mathbb{R}^n}\frac{(z(x)-z(y))((z-k)^+(x)-(z-k)^+(y))}{|x-y|^{n+2s}} d x d y
         \leq \int_{\{z>l\}} (z-l)^+dx.
    \end{equation*}
    One can now use %the well-known
    Stampacchia's method [\citealp{stampacchia}, Lemma 4.1] to deduce $\|z\|_{L^\infty(\Omega)}\leq T$ for some constant $T>0$. \\
    For $0<\lambda\neq 1$, let $0<w\in W^{1,2}_0(\Omega)$ satisfies \cref{bddine}. Taking $z=\left(\frac{1}{\lambda}\right)^{\frac{1}{\gamma+1}}w\in W^{1,2}_0(\Omega)$, we get by \cref{bddine}, if $\phi>0$,
    \begin{equation*}
    \begin{array}{l}
        \quad  \int_\Omega \nabla z\cdot \nabla \phi \, d x+\int_{\mathbb{R}^n}\int_{\mathbb{R}^n}\frac{(z(x)-z(y))(\phi(x)-\phi(y))}{|x-y|^{n+2s}} d x d y\\=\left(\frac{1}{\lambda}\right)^{\frac{1}{\gamma+1}} \left(\int_\Omega \nabla w\cdot \nabla \phi \, d x+\int_{\mathbb{R}^n}\int_{\mathbb{R}^n}\frac{(w(x)-w(y))(\phi(x)-\phi(y))}{|x-y|^{n+2s}} d x d y\right)\leq \left(\frac{1}{\lambda}\right)^{\frac{1}{\gamma+1}}\int_\Omega \frac{\lambda}{w^\gamma}\phi\, dx.%\quad \forall \phi\in H^1_0(\Omega).
    \end{array}
    \end{equation*}
    This implies
     \begin{equation*}
    \int_\Omega \nabla z\cdot \nabla \phi \, d x+\int_{\mathbb{R}^n}\int_{\mathbb{R}^n}\frac{(z(x)-z(y))(\phi(x)-\phi(y))}{|x-y|^{n+2s}} d x d y\leq \int_\Omega \frac{1}{z^\gamma}\phi\, dx,\quad \forall \phi\in W^{1,2}_0(\Omega),\, \phi>0,
    \end{equation*}
    and hence $\|z\|_{L^\infty(\Omega)}\leq T$, which gives $\|w\|_{L^\infty(\Omega)}\leq T\lambda^{\frac{1}{\gamma+1}}$.  
    \end{proof}
We will follow \cite{boccardoexistence} and construct suitable sub and supersolutions to the approximated problem 
\begin{equation}{\label{approximated2}}
\begin{split}
-\Delta u_k+(-\Delta)^s u_k&=\frac{\lambda}{(u_k^++\frac{1}{k})^{\gamma}}+(u_k^+)^r  \text { in } \Omega, %\\u&>0 \text{ in } \Omega.
\\u_k&=0  \text { in }\mathbb{R}^n \backslash \Omega.
\end{split}
\end{equation}
One can see [\citealp{garain}, Lemma 3.2] to find a unique positive $w_k\in W^{1,2}_0(\Omega)$, for each $k\in \mathbb{N}$, which is a weak solution to
\begin{equation*}{\label{approximated3}}
\begin{split}
-\Delta w_k+(-\Delta)^s w_k&=\frac{\lambda}{(w_k^++\frac{1}{k})^{\gamma}}%+(w_k^+)^r 
\text { in } \Omega, %\\u&>0 \text{ in } \Omega,
\\w_k&=0  \text { in }\mathbb{R}^n \backslash \Omega,
\end{split}
\end{equation*}
and satisfies for each $\omega\subset\subset\Omega$, $w_{k}\geq c(\omega)>0$ in $\omega$, for some $c\equiv c(\omega)$. Also, $\{w_k\}$ is monotonically increasing in $k$. As each $w_k$ is positive, so it solves 
$%\begin{equation*}
%\begin{split}
-\Delta w_k+(-\Delta)^s w_k=\frac{\lambda}{(w_k+\frac{1}{k})^{\gamma}}%+(w_k^+)^r 
\text { in } \Omega, %\\u&>0 \text{ in } \Omega,
%\\w_k&=0  \text { in }\mathbb{R}^n \backslash \Omega,
%\end{split}
%\end{equation*}
$ and $\frac{\lambda}{(w_k+\frac{1}{k})^{\gamma}}%+(w_k^+)^r
\leq \frac{\lambda}{(w_k+\frac{1}{k})^{\gamma}}+(w_k)^r$ implies that $w_k\in W_0^{1,2}(\Omega)$ is a subsolution to \cref{approximated2}.\smallskip\\
In order to construct a supersolution, we take $z_{k,t}\in W^{1,2}_0(\Omega)$ to be the unique positive weak solution of 
%\begin{equation*}
%\begin{split}
%-\Delta z_{k,t}+(-\Delta)^s z_{k,t}&=\frac{t}{(z_{k,t}^++\frac{1}{k})^{\gamma}}%+(w_k^+)^r 
%\text { in } \Omega. %\\u&>0 \text{ in } \Omega,
%\\w_k&=0  \text { in }\mathbb{R}^n \backslash \Omega,
%\end{split}
%\end{equation*}
%The positivity of $z_{k,t}$ guarantees that $z_{k,t}$ satisfies 
\begin{equation*}
\begin{split}
-\Delta z_{k,t}+(-\Delta)^s z_{k,t}&=\frac{t}{(z_{k,t}+\frac{1}{k})^{\gamma}}%+(w_k^+)^r 
\text { in } \Omega. %\\u&>0 \text{ in } \Omega,
%\\w_k&=0  \text { in }\mathbb{R}^n \backslash \Omega,
\end{split}
\end{equation*}
By \cref{boundednessofsubsolution}, we get the existence of $M>0$ such that
\begin{equation*}
    \|z_{k,t}\|_{L^\infty(\Omega)}\leq Mt^{\frac{1}{\gamma+1}}.
\end{equation*}
Starting from here, one can follow [\citealp{boccardoexistence}, Theorem 2.1, Step 3] to find the existence of $\Lambda>0$ such that for each $0<\lambda< \Lambda$, there exists $T\equiv T(\lambda)> \lambda >0$; and $z_{k,t}$ is a supersolution to \cref{approximated2}, $\forall \, t\geq T(\lambda)$ for large enough $k$. \smallskip\\Now we show $w_k \leq z_{k, t}$. Indeed
\begin{equation*}
-\Delta(w_k-z_{k, t})+(-\Delta)^s(w_k-z_{k, t})=\frac{\lambda}{(w_k+\frac{1}{k})^\gamma}-\frac{t}{(z_{k, t}+\frac{1}{k})^\gamma}  , \text{ in } \Omega.  
\end{equation*}
Taking $(w_k-z_{k,t})^+$ as a test function, we obtain
\begin{equation*}
    \begin{array}{l}
\int_\Omega|\nabla (w_k-z_{k,t})^+|^2dx+\int_{\mathbb{R}^n}\int_{\mathbb{R}^n}\frac{((w_k-z_{k,t})(x)-(w_k-z_{k,t})(y))((w_k-z_{k,t})^+(x)-(w_k-z_{k,t})^+(y))}{|x-y|^{n+2s}}dxdy\smallskip\\= \int_{\Omega}\left[\frac{\lambda}{(w_k+\frac{1}{k})^\gamma}-\frac{t}{(z_{k, t}+\frac{1}{k})^\gamma}\right](w_k-z_{k, t})^{+}dx
= \int_{\Omega}\left[\frac{\lambda}{(w_k+\frac{1}{k})^\gamma}-\frac{\lambda}{(z_{k, t}+\frac{1}{k})^\gamma}\right](w_k-z_{k, t})^{+} d x\smallskip\\\quad\quad\quad\quad\quad\quad\quad\quad\quad\quad\quad\quad\quad\quad\quad\quad\quad\quad\quad\quad\quad\quad+\int_{\Omega}\left[\frac{\lambda}{(z_{k,t}+\frac{1}{k})^\gamma}-\frac{t}{(z_{k, t}+\frac{1}{k})^\gamma}\right](w_k-z_{k, t})^{+}dx
 \smallskip\\
\quad\quad\quad\quad\quad\quad\quad\quad\quad\quad\quad\quad\quad\quad\quad\quad\quad\quad\quad\quad\quad= \lambda \int_{\left\{w_k \geq z_{k, t}\right\}}\left[\frac{1}{(w_k+\frac{1}{k})^\gamma}-\frac{1}{(z_{k, t}+\frac{1}{k})^\gamma}\right](w_k-z_{k, t})dx\smallskip\\
\quad\quad\quad\quad\quad\quad\quad\quad\quad\quad\quad\quad\quad\quad\quad\quad\quad\quad\quad\quad\quad\quad+\int_{\left\{w_k \geq z_{k, t}\right\}} \frac{1}{(z_{k,t}+\frac{1}{k})^\gamma}[\lambda-t]\left(w_k-z_{k, t}\right)dx .         
    \end{array}
\end{equation*}
Note that, the second term on the left in the first line is nonnegative; in the last line, the first integral is negative due to the monotonicity of $v\to \frac{1}{\left(v+\frac{1}{k}\right)^\gamma}$ for $v>0$ and the second integral is negative since $t \geq T(\lambda)>\lambda$. This readily gives $w_k\leq z_{k,t}$.
\\One now can define $g(v)=\frac{\lambda}{\left(v+\frac{1}{k}\right)^\gamma}+v^p+k^{\gamma+1} \lambda \gamma v, v \in[0, \infty)$ and note that $g$ is increasing. Further, observing that for any two functions $\phi,\psi\in W^{1,2}_0(\Omega)$, the integral \begin{equation*}\int_{\mathbb{R}^n}\int_{\mathbb{R}^n}\frac{((\phi-\psi)(x)-(\phi-\psi)(y))((\phi-\psi)^+(x)-(\phi-\psi)^+(y))}{|x-y|^{n+2s}}dxdy\end{equation*}is nonnegative, one can use the classical subsolution-supersolution technique (see [\citealp{LCE}, Chapter 9]) to get the existence of $u_k\in W^{1,2}_0(\Omega)$, which is a solution of
\cref{approximated2} and satisfies
\begin{equation}{\label{unibdd}}
w_k \leq u_k \leq z_{k, t} \leq M t^{\frac{1}{\gamma+1}} .    
\end{equation}
Again the positivity of the increasing sequence $\left\{w_k\right\}$ inside  compactly contained subsets of $\Omega$ gaurantees that for every $\omega \subset \subset \Omega$ there exists $c_\omega>0$ (independent of $k$ ) such that
\begin{equation}{\label{unibdd2}}
u_k\geq w_k \geq w_1 \geq c_\omega>0, \quad \text { in }  \omega, \text { for every } k \in \mathbb{N} .    
\end{equation}
For $0<\gamma\leq 1$, we can now choose $u_k$ as a test function in \cref{approximated2} and use \cref{unibdd} to get $\left\{u_k\right\}$ is bounded in $W^{1,2}_0(\Omega)$. Then $\exists\,u \in W^{1,2}_0(\Omega) \cap L^{\infty}(\Omega)$ such that up to a subsequence  $u_{k}\rightharpoonup u$ in $W^{1,2}_0(\Omega)$ and a.e. to $u \geq w_1>0$ in $\Omega$.
Furthermore we have, for $\phi$ in $W^{1,2}_0(\omega)$,
\begin{equation*}
    0 \leq \frac{\lambda \phi}{\left(u_k+\frac{1}{k}\right)^\gamma} \leq \frac{\lambda| \phi|}{\left(c_\omega\right)^\gamma} \text{ and } u_k^{r}\phi\leq\left( M t^{\frac{1}{\gamma+1}} \right)^{r}|\phi| .
\end{equation*}
Therefore, by the Dominated convergence theorem, one has
\begin{equation*}
\lim _{k \rightarrow\infty} \int_{\Omega} \frac{\lambda \phi}{\left(u_k+\frac{1}{k}\right)^\gamma}=\lambda \int_{\Omega} \frac{\phi}{u^\gamma}   \text { and } \lim _{k \rightarrow\infty} \int_{\Omega} u_k^r\phi \,dx=\int_{\Omega} u^r\phi \,dx. 
\end{equation*}
This along with $u_k\rightharpoonup u$ in $W^{1,2}_0(\Omega)$ assures that $u$ is a weak solution to \cref{prob2} for the case $0<\gamma\leq 1$. \smallskip\\
For $\gamma>1$, as $u_k$ is bounded, one can choose $u_k^\gamma$ as a test function in \cref{approximated2} to get $\{u_k^{\frac{\gamma+1}{2}}\}$ is bounded in $W^{1,2}_0(\Omega)$. Indeed, we get by using \cref{unibdd}
    \begin{equation}{\label{posi}}
\int_{\Omega} \nabla u_k \cdot\nabla u_k^\gamma\, d x+\int_{\mathbb{R}^n}\int_{\mathbb{R}^n}\frac{(u_k(x)-u_k(y))(u_k^\gamma(x)-u_k^\gamma(y))}{|x-y|^{n+2s}} d x d y=\lambda\int_\Omega \frac{u_k^{\gamma}}{(u_k+\frac{1}{k})^\gamma} d x +\int_\Omega u^{\gamma+r}\,d x %\quad \forall \phi \in W_0^{1,2}(\omega)
\leq C.
\end{equation}
Here we can use item $(i)$ of \cref{algebraic} to get
 \begin{equation*}
\frac{4\gamma}{(\gamma+1)^2}\int_{\Omega} |\nabla u_k^{\frac{\gamma+1}{2}}|^2 d x+\frac{4\gamma}{(\gamma+1)^2}\int_{\mathbb{R}^n}\int_{\mathbb{R}^n}\frac{(u_k^{\frac{\gamma+1}{2}}(x)-u_k^{\frac{\gamma+1}{2}}(y))^2}{|x-y|^{n+2s}} d x d y%\leq\lambda\int_\Omega u_k^{\gamma-1} d x +\int_\Omega u^{\gamma+r}\,d x %\quad \forall \phi \in W_0^{1,2}(\omega)
\leq C.
\end{equation*}
The above readily implies $\{u_k^{\frac{\gamma+1}{2}}\}$ is bounded in $W^{1,2}_0(\Omega)$. Since $\gamma>1$, and $\Omega$ is bounded, and $\left\{u_k\right\}_k$ is uniformly bounded in $L^{\gamma+1}(\Omega)$, we deduce that $\left\{u_k\right\}_k$ is uniformly bounded in $L^2(\Omega)$, in particular in $L^2(K )$, for every subset $K$ compactly contained in $ \Omega$. Further, as $K\times K \subset \Omega\times\Omega \subset \mathbb{R}^{2n}$ and all the integrals in the left-hand-side of \cref{posi} are positive, hence we have,
\begin{equation*}
         \int_K \int_K \frac{(u_k(x)-u_k(y))(u_k^\gamma(x)-u_k^\gamma(y))}{|x-y|^{n+2 s}} d x d y \leq C
    \text{ and }
   \int_K  u_k^{\gamma-1}|\nabla u_k|^2 d x \leq C,
\end{equation*}
for every $K \subset\subset \Omega$. We now apply the item $(iii)$ of \cref{algebraic}, to get
\begin{equation*}
    \int_K \int_K \frac{\left| u_k(x)-u_k(y)\right|^2\left|u_k(x)+u_k(y)\right|^{\gamma-1}}{|x-y|^{n+2 s}} d x d y  \leq C.
\end{equation*}
Using the positivity of $u_k$ in $K$ for all $k$ (see \cref{unibdd2}), one now gets
\begin{equation}{\label{est1}}
 \int_K \int_K \frac{\left| u_k(x)-u_k(y)\right|^2}{|x-y|^{n+2 s}} d x d y \leq \frac{2^{1-\gamma} C_\gamma}{c_{K%[t_1,t_2]
}^{\gamma-1}} \text{ and }  \int_K |\nabla u_k|^2d x d y \leq \frac{C_\gamma}{ c_{K}^{\gamma-1}}.
\end{equation}
Hence $\left\{u_k\right\}_k$ is uniformly bounded in $W_{\operatorname{loc}}^{1,2}(\Omega)$. So $u_k\rightharpoonup u$ in $W_{\operatorname{loc}}^{1,2}(\Omega)$. One can now follow [\citealp{garain}, Theorem 2.13], [\citealp{scase}, Theorem 3.6] and the same argument of the case $0<\gamma\leq 1$ to deduce that u is a weak solution of \cref{prob2}.
\begin{remark}
    The cases $0<r<1$ and $r=1$ are same as \cite{boccardoexistence}, we refer to Theorem 2.4 and Remark 2.5 therein.
\end{remark}\begin{remark}
    \cref{mainth2} does not require the conditions $s\in(0,1/2)$, $n>2$, $r<2^*-1$ and smoothness of the boundary of $\Omega$ or convexity of the domain. This result also interprets the fact $u=0$ in $\mathbb{R}^n\backslash\Omega$ in the sense that some powers of $u$ is in $W^{1,2}_0(\Omega)$.  
\end{remark}
\begin{remark}{\label{nonexistence}}
 We remark that there exists a positive number $\overline{\Lambda}$ independent of $k$ such that \cref{prob2} has no solution if $\lambda\geq \overline{\Lambda}$.  Let $\omega \subset \subset \Omega$ and $e_1^\omega$ be the first (positive) eigenfunction (see [\citealp{eigenvalue}, Proposition 5.1]) of
\begin{equation*}
\begin{split}
-\Delta e_1^\omega+(-\Delta)^s e_1^\omega&=\lambda_1^\omega e_1^\omega \text { in } \omega, \\e_1^\omega &>0 \text{ in } \omega,\\e_1^\omega&=0  \text { in }\mathbb{R}^n \backslash \omega.
\end{split}
\end{equation*}
Clearly in $\Omega$ we have $-\Delta e_1^\omega+(-\Delta)^s e_1^\omega \leq \lambda_1^\omega e_1^\omega$. Now if there exists a solution to \cref{prob2}, for every $\lambda>0$, we use $e_1^\omega$ as a test function in \cref{prob2} to obtain
\begin{equation*}
\lambda \int_{\Omega} \frac{e_1^\omega}{u^\gamma}dx+\int_{\Omega} u^r e_1^\omega dx \leq \lambda_1^\omega \int_{\Omega} u e_1^\omega dx .    
\end{equation*}
One can now use Young's inequality to get
\begin{equation*}
   \lambda \int_{\Omega} \frac{e_1^\omega}{u^\gamma}dx+\int_{\Omega} u^r e_1^\omega dx \leq  \frac{1}{2} \int_{\Omega} u^r e_1^\omega dx+ C_1(\lambda_1)^{r^\prime}\int_\Omega e_1^\omega dx,
\end{equation*}
that is
\begin{equation*}
    \int_{\Omega}\left[\frac{\lambda}{u^\gamma}+\frac{1}{2} u^r-C_1(\lambda_1^\omega)^{r^{\prime}}\right] e_1^\omega dx \leq 0 .
\end{equation*}
Since the real function $\frac{\lambda}{t^\gamma}+\frac{1}{2} t^r-C_1(\lambda_1^\omega)^{r^{\prime}}, t>0$, is greater than $C_2 \lambda^{\frac{r}{r+\gamma}}-C_1(\lambda_1^\omega)^{r^{\prime}}$, for some $C_2>0$, the last inequality implies that
\begin{equation*}
    \int_{\Omega}\left[C_2 \lambda^{\frac{r}{r+\gamma}}-C_1\left(\lambda_1^\omega\right)^{r^{\prime}}\right] e_1^\omega\leq 0,
\end{equation*}
which is impossible for $\lambda$ large enough. This contradiction proves the nonexistence for large $\lambda$.
\end{remark}
\section*{Preliminaries for the proof of \cref{mainth3}} We have considered $\Omega$ to be a bounded strictly convex domain with smooth boundary, $1<r<2^*-1=\frac{n+2}{n-2}$ and $s\in(0,1/2)$. As before, for $k\in \mathbb{N}$ we take the approximated problems
\begin{equation}{\label{approximated4}}
\begin{split}
-\Delta u+(-\Delta)^s u&=\lambda f_k(u)+g(u)%\frac{\lambda}{(u_k^++\frac{1}{k})^{\gamma}}+(u_k^+)^r 
\text { in } \Omega, 
\\u&>0 \text{ in } \Omega, 
\\u&=0  \text { in }\mathbb{R}^n \backslash \Omega,
\end{split}
\end{equation}
where $f_k$ and $g$ are the continuous functions given by
\begin{equation*}
    f_k(t)=\frac{1}{(t+\frac{1}{k})^\gamma}, \quad g(t)=t^r, 1<r<2^*-1; \quad \text { for } t \geq 0 .
\end{equation*}
Clearly, being continuously differentiable, $f_k$ and $g$ are locally Liptchitz and they satisfy the following properties:
\begin{equation}{\label{property1}}
  g(0)=0,\quad \lim _{t \rightarrow 0^{+}} \frac{g(t)}{t}=0 \quad \text { and }\quad \lim _{t \rightarrow 0^{+}} \frac{f_k(t)}{t}=+\infty, \\
\end{equation}
\begin{equation}{\label{property2}}
 \lim _{t \rightarrow+\infty} \frac{g(t)}{t^r}=1>0, \quad
 \lim _{t \rightarrow+\infty} \frac{f_k(t)}{t^r}=0, \quad \text { uniformly in } k, \\
    \end{equation}
\begin{equation}{\label{property3}}
    \frac{\lambda f_k(t)+g(t)}{t^\sigma} \text { is nonincreasing for } t \geq 0, \quad \text { with } \sigma=\frac{n+2}{n-2} .
\end{equation} 
Further, for the non-perturbed case, we consider $w_{k,\lambda}\in W^{1,2}_0(\Omega)\cap L^\infty(\Omega)$, for each $k\in \mathbb{N}$, to be the unique positive weak solution (see [\citealp{garain}, Lemaa 3.2]) to
\begin{equation}{\label{approximated5}}
\begin{split}
-\Delta w+(-\Delta)^s w&=\frac{\lambda}{(w+\frac{1}{k})^{\gamma}}%+(w_k^+)^r 
\text { in } \Omega, \\w&>0 \text{ in } \Omega,
\\w&=0  \text { in }\mathbb{R}^n \backslash \Omega.
\end{split}
\end{equation}
The sequence $\{w_{k,\lambda}\}$ is monotonically increasing with respect to $k$ and for every $\omega\subset\subset\Omega$, there exists a constant $c\equiv c(\omega)>0$, independent of $k$ such that
\begin{equation}{\label{greater}}
    w_{k,\lambda}\geq c(\omega)>0.
\end{equation} Note that, each $w_{k,\lambda}$ satisfies hypothesis of \cref{boundednessofsubsolution}.
\begin{remark}{\label{delta}}
    We make an important remark that if $\lambda=0$ and $k \in \mathbb{N}$, then the problem \cref{approximated4} does not depend on $k$ and every solution of it satisfies
\begin{equation*}
    \|u\|_{L^\infty(\Omega)}>\delta_0 \text{ for some } \delta_0>0.
\end{equation*}
Indeed, by \cref{property1}, we can select $\delta_0>0$ such that $g(t)<\lambda_1 t$ for all $t \in\left[0, \delta_0\right]$, where $\lambda_1$ is the first eigenvalue and let $e_1$ be the first positive eigenfunction (see [\citealp{eigenvalue}, Proposition 5.1]) corresponding to the equation
\begin{equation*}
-\Delta v+(-\Delta)^s v=\lambda_1 v \text { in } \Omega,% \quad v>0 \text{ in } \Omega, 
\quad v=0 \text { in } \mathbb{R}^n \backslash \Omega .    
\end{equation*}
Now taking $e_1$ as test function in \cref{approximated4} with $\lambda =0$, we get
\begin{equation*}
    \int_{\Omega}\left(\lambda_1 u-g(u)\right) e_1=0.
\end{equation*}
and consequently, $\|u\|_{L^\infty(\Omega)}>\delta_0$.
\end{remark}
\begin{remark}{\label{contin}}
    Every weak $W^{1,2}_0(\Omega)$ solution of \cref{approximated4,approximated5} are indeed classical $ {C}^2$ solution thus validating our calculations. We use the so-called bootstrap argument to show that $u$ (weak solution of \cref{approximated4}) is Hölder continuous on $\overline{\Omega}$.  Set $\kappa r=(n+2) /(n-2)$ and note that $\kappa>1$.
    \smallskip\\ As $u\in W^{1,2}_0(\Omega)$, by Sobolev embedding $u \in L^{2^*%\frac{2n}{(n-2)}
    }(\Omega)$. Further one can see [\citealp{ambrosetti}, Chapter 1, Section 2] and get that as $|\lambda f_k(t)+g(t)|\leq a_{k}+|t|^r$, for some $a_{k}>0$, the Nemitski operator corresponding to $\lambda f_k+g$ is continuous from $L^\alpha(\Omega)$ to $L^\beta(\Omega)$, where $r=\alpha/\beta.$ Using this with $\alpha=2^*$ we infer that $\lambda f_k( u)+g(u) \in L^\beta(\Omega)$ with $\beta=\frac{2^*}{r}  =\frac{2\kappa n}{(n+2)}$. From here, one can use [\citealp{valdinoci}, Theorem 1.4] to get $u \in W^{2, \beta}(\Omega)$. If $2 \beta>n$ then $u \in C^{0, \eta}(\overline{\Omega})$. \\Otherwise, we can repeat the above steps:\\
(i) by Sobolev embedding theorem one has that $u \in L^{q^{\prime}}(\Omega)$, with
\begin{equation*}
    q^{\prime}=\frac{n \beta}{n-2 \beta}>\kappa \frac{2 n}{n-2} ;
\end{equation*}
(ii) it follows that $\lambda f_k( u)+g(u) \in L^{\beta^{\prime}}(\Omega)$ with $\beta^{\prime}=q^{\prime} / r>\kappa \beta$;\\
(iii) another use of [\citealp{valdinoci}, Theorem 1.4] gives $u \in W^{2, \beta^{\prime}}(\Omega)$.
\smallskip\\In any case, after a finite number of times, one finds that $u \in W^{2, q}(\Omega)$ with $2q>n $. Then the Sobolev embedding theorem yields $W^{2, q}(\Omega) \subset C^{0, \eta}(\overline{\Omega})$ with some $\eta < 1$.
\smallskip\\At this point we can apply [\citealp{FaberKrahn}, Theorem 2.8]. Indeed, letting $h(x)=\lambda f_k( u(x))+g(u(x)), u$ is a weak solution of $-\Delta u+(-\Delta)^su=h$ with $h \in C^{0, \eta}(\overline{\Omega})$ for some $\eta < 1$. Hence $u \in C^{2, \eta}(\overline{\Omega})$ and is a classical solution of \cref{approximated4}.\\ For a solution of \cref{approximated5}, one does not need this bootstrap argument as $\frac{\lambda}{(t+\frac{1}{k})^\gamma}\leq \lambda k^\gamma$, for $t\geq 0$ and hence $w_{k,\lambda}\in W^{2,q}(\Omega)$ for all $1<q<\infty$, so is H\"older and [\citealp{FaberKrahn}, Theorem 2.8] can be applied directly.
\smallskip\\
    If we take weak solutions belonging to $C(\overline{\Omega})$ of \cref{approximated4,approximated5}, then we can skip the $W^{2,q}$ estimate of \cite{valdinoci}. Indeed, in this case $g(u)$ is bounded (with $f_k(u)$ is always bounded) and one can apply [\citealp{MG}, Theorem 4] to get $u,w_{k,\lambda}\in C^{0,\alpha}(\mathbb{R}^n)$ for all $\alpha\in (0,1)$ and in particular $u,w_{k,\lambda}\in C^{0,\alpha}(\overline{\Omega})$. As $f_k$ and $g$ are locally Lipschitz, so this gives $\lambda f_k(u)+g(u)\in C^{0,\alpha}(\overline{\Omega})$ and we can directly use [\citealp{FaberKrahn}, Theorem 2.8].
\end{remark}
\begin{remark}{\label{subsuper}}
We denote by $C^m_0(\overline{\Omega})$ to be the space of $m$ times differentiable functions vanishing in $\mathbb{R}^n\backslash\Omega$. Let $w_{k,\lambda}\neq v\in C^2_0(\overline{\Omega})$ be a supersolution of \cref{approximated5}. Then choosing $(w_{k,\lambda}-v)^+$ as a test function, we get 
\begin{equation*}
    \begin{array}{l}
\int_\Omega|\nabla (w_{k,\lambda}-v)^+|^2dx+\int_{\mathbb{R}^n}\int_{\mathbb{R}^n}\frac{((w_{k,\lambda}-v)(x)-(w_{k,\lambda}-v)(y))((w_{k,\lambda}-v)^+(x)-(w_{k,\lambda}-v)^+(y))}{|x-y|^{n+2s}}dxdy\smallskip\\\leq \int_{\Omega}\left[\frac{\lambda}{(w_{k,\lambda}+\frac{1}{k})^\gamma}-\frac{\lambda}{(v+\frac{1}{k})^\gamma}\right](w_{k,\lambda}-v)^{+}dx.         
    \end{array}
\end{equation*}
Noting that the second integral on the left is nonnegative and the integral on the right is negative on the set $\{w_{k,\lambda}>v\}$, we get 
\begin{equation*}
    \int_\Omega|\nabla (w_{k,\lambda}-v)^+|^2dx=0.
\end{equation*}
This along with the continuity of $w_{k,\lambda}$ and $v$, we get $w_{k,\lambda}\leq v$ in $\Omega$. Now let $z=v-w_{k,\lambda}$, then if $M$ is the Lipschitz constant of $f_k$ in $[\underset{\bar{\Omega}}{\min} \, w_{k,\lambda}, \underset{\bar{\Omega}}{\max} \, v]$, we have 
\begin{equation*}
  -  \Delta z(x)+(-\Delta)^s z(x)+\lambda M z(x)\geq \lambda f_k(v(x))-\lambda f_k(z_{k,\lambda}(x))+\lambda M z(x)\geq 0 \quad \text{ pointwise in } \Omega. 
\end{equation*}
We now apply strong maximum principle [\citealp{Biagiregularitymaximum}, Theorem 1.3] (this theorem can be easily generalized for the operator $-\Delta +(-\Delta)^s+L$, where $L$ is a positive number), to get $v> w_{k,\lambda}$ in $\Omega$. One now can proceed as [\citealp{FaberKrahn}, Theorem 2.9] and use the classical Hopf Lemma for $-\Delta +L$, $L$ is positive, to conclude
\begin{equation*}
   \frac{\partial v}{\partial \eta}<  \frac{\partial w_{k,\lambda}} {\partial \eta} \text { on } \partial \Omega,
\end{equation*}
where $\eta$ is the outward unit normal to the boundary $\partial \Omega.$ The same conclusion holds for a subsolution. In particular $0$ is a subsolution to \cref{approximated5}, and any solution to \cref{approximated4} satisfies 
\begin{equation*}
    \frac{\lambda}{(u+\frac{1}{k})^\gamma}+u^r\geq \frac{\lambda}{(u+\frac{1}{k})^\gamma},
\end{equation*}
and hence is a supersolution to \cref{approximated5}. Therefore it holds
\begin{equation*}
  0<w_{k,\lambda}<u \quad \text{ and }\quad \frac{\partial u}{\partial \eta}<  \frac{\partial w_{k,\lambda}} {\partial \eta}<0 \text { on } \partial \Omega.
\end{equation*}
\end{remark}
We now prove the uniform apriori estimate regarding solutions of \cref{approximated4}.
\begin{lemma}{\label{uniformboundedness}}
    For every $\lambda>0$, there exists $M>0$, independent of $k$, such that every positive $u\in C(\mathbb{R}^n)$, satisfying \cref{approximated4}, also satisfies $\|u\|_{L^\infty(\Omega)}< M$.
\end{lemma}
\begin{remark}
    As our domain is strictly convex with smooth boundary, we will use the classical moving plane argument. We take $\zeta$ to be a unit vector in $\mathbb{R}^n$ and let $T_\alpha$ denote the hyperplane $\zeta \cdot x=\alpha$. For $\alpha=\tilde{\alpha}$ large, $T$ is disjoint from $\overline{\Omega}$. Now we start moving the plane continuously toward $\Omega$, preserving the same normal, i.e., decrease $\alpha$, until it begins to intersect $\overline{\Omega}$. From that moment on, at every stage the plane $T_\alpha$ will cut off from $\Omega$ an open cap $\Sigma(\alpha)$, the part of $\Omega$ which is on the same side of $T_\alpha$ as $T_{\tilde{\alpha}}$. Let $\Sigma^{\prime}(\alpha)$ be the reflection of $\Sigma(\alpha)$ in the plane $T_\alpha$. Hence at the beginning, $\Sigma^{\prime}(\alpha)$ will lie in $\Omega$ and as $\alpha$ decreases, the reflected cap $\Sigma^{\prime}(\alpha)$ will remain in $\Omega$, at least until one of the following occurs:\\
(i) $\Sigma^{\prime}(\alpha)$ becomes internally tangent to $\partial \Omega$ at some point $P$ not on $T_\alpha$ or\\
(ii) $T_\alpha$ reaches a position where it is orthogonal to the boundary of $\Omega$ at some point $Q$.
\\Following the notations of \cite{GidasNiNirenberg}, we denote by $T_{\alpha_1}: \zeta \cdot x=\alpha_1$ the plane $T_\alpha$ when it first reaches one of these positions and we call $\Sigma(\alpha_1)=\Sigma_\zeta
$ the maximal cap associated with $\zeta$. Note that its reflection $\Sigma_\zeta^{\prime}$ in $T_{\alpha_1}$ lies in $\Omega$. By \cite{GidasNiNirenberg}, the existence of such maximal caps is always guaranteed.
\end{remark}
\begin{proof}
    Let $\lambda >0$ be fixed. We will prove the result in two steps. First we will show the existence of an open set $\omega_0\subset\subset\Omega$ such that 
    \begin{equation}{\label{step1movingplane}}
        u(x)\leq \underset{\bar{\omega}_0}{\max}\, u,\quad\forall x\in \Omega\backslash\omega_0;
    \end{equation}
    and in the second step, we will find a positive constant $M_{\omega_0}$ such that 
    \begin{equation}{\label{step2movingplane}}
         \underset{\bar{\omega}_0}{\max} \,u\leq M_{\omega_0}
    \end{equation}
    for every positive solution $u$ of \cref{approximated4}.\smallskip\\
   \textbf{Step $1$:} Fix $x\in \partial \Omega$. Let $\eta(x)$ be the outward unit normal vector at $x$ to $\partial \Omega$ and let $T_x$ be the tangent hyperplane at $x$ to $\overline{\Omega}$. As $\Omega$ is strictly convex with smooth boundary, therefore this hyperplane is a supporting hyperplane (in fact unique supporting hyperplane at $x$) and divides the space $\mathbb{R}^n$ into two disjoint half-spaces with $\Omega$ belonging to one of them. Further, $\overline\Omega$ intersects $T_x$ only at $x$. Hence we can consider another hyperplane $T$ which is parallel to $T_x %J_x(\bar{\Omega})
   $ and cuts off $\overline{\Omega}$ in such a way that if one reflects $\Sigma_x$ (the region of ${\Omega}$ which is between both hyperplanes) with respect to $T$, then $\Sigma^\prime_x$ (the reflection) is inside $\Omega$.
\smallskip\\By strict convexity of $\Omega$, we can choose $t_x$ and $\varepsilon_x$, two positive numbers such that if $y \in \partial \Omega \cap \bar{B}(x, \varepsilon_x)$ then, $y-t \eta(x) \in \Omega \cap \Sigma_x$ for all $t\in\left(0, t_x\right]$. In spirit of \cite{Arcoyamultigreaterthan}, we define the sets
\begin{equation*}
    \begin{array}{c}
         V_x=\{y-s\eta (x): y \in \partial\Omega \cap B(x, \varepsilon_x), 0 \leq t<t_x\} \subset \overline{\Omega} \text{ and }\\
{W}_x=\{y-t_x \eta(x): y \in \partial \Omega \cap \bar{B}(x, \varepsilon_x)\} \subset \Omega \cap \Sigma_x .
    \end{array}
\end{equation*}
One can easily check that $V_x$ is open in $\overline{\Omega}$ (in subspace topology) and ${W}_x$ is compact in $\Omega$. Consequently $\delta_x=\operatorname{dist}\left(W_x, \partial \Omega\right)>0$. Now by the compactness of $\partial \Omega$, one can extract a finite subcover from the covering $\cup_{x \in \partial \Omega} V_x$, that is, there exists $x_1, x_2, \ldots, x_k\in \partial \Omega$ such that $\partial \Omega \subset \cup_{i=1}^k V_i$, where we denote $V_i=V_{x_i}$ with $ t_i=t_{x_i}, \varepsilon_i=\varepsilon_{x_i}$, $W_i=W_{x_i}$, $\delta_i=\delta_{x_i}$ and $\Sigma_i=\Sigma_{x_i}$. \\As $V_i$ is open in $\overline{\Omega}$ containing $\partial \Omega \cap B(x_i, \varepsilon_i)$ (set $t=0$), the set $V:=\cup_{i=1}
^k V_i$ is an open neighborhood of $\partial \Omega$ in $\overline{\Omega}$. Therefore, $\operatorname{dist}(\partial \Omega, \overline{\Omega} \backslash V)=d>0$. Indeed if there exists $\{v_k\} \subset \partial \Omega$ and $\{z_k\} \subset \overline{\Omega} \backslash V$ such that $\mathrm{d}\left(v_k, z_k\right) \rightarrow 0$, then up to a subsequence $v_k \rightarrow v \in \partial \Omega$ (Bolzano-Weirestrass) and $z_k \rightarrow v \in \partial \Omega$ contradicting that $V$ is a neighborhood of $\partial \Omega$. Taking $r=\min \{d, \delta_1, \ldots, \delta_k\}>0$ and we consider the open set $\omega_0:=\{y \in \overline{\Omega}: \operatorname{dist}(y, \partial \Omega)>\frac{r}{2}\}$. Clearly $\Omega \backslash \omega_0 \subset V$ and $W_i \subset \omega_0$ for all $i=1, \ldots, k$.
\smallskip\\Now let $u$ is a positive solution of \cref{approximated4}. Take $x \in \Omega \backslash \omega_0 \subset V$. So $\exists \,i \in\{1, \ldots, k\}$ such that $x \in V_i$, i.e. there exists $t \in(0, t_i)$ and $y_i \in \partial{\Omega} \cap B(x_i, \varepsilon_i)$ such that $x=y_i-t \eta(x_i)$.\smallskip\\
Since $\lambda f_k+g$ is locally Lipschitz, at this point, one can use [\citealp{Biagisymmetry}, Theorem 1.1]. Indeed, one can take the unit vector $\eta(x_i)$ instead of $(1,0,\ldots, 0)$ and end up with the same result, but only up to the maximal cap, which is easy to deduce for a convex nonsymmetric domain if one follows the proof of \cite{Biagisymmetry}. In particular if one takes $\underset{x\in\Omega}{\operatorname{inf}}\,x_1=a$, $\underset{x\in\Omega}{\operatorname{sup}}\,x_1=b$, $\{x\in \mathbb{R}^n; x_1=\lambda_1\}$ to be the hyperplane corresponding to the maximal cap in $-x_1$ direction and $\{x\in \mathbb{R}^n; x_1=\lambda_2\}$ to be the hyperplane corresponding to the maximal cap in $x_1$ direction, then Lemma 2.8 of \cite{Biagisymmetry} can be deduced for $\lambda\in (a,\lambda_1)\cup(\lambda_2,b)$. Further, choosing $l=b+\lambda-\lambda_1$ and $\eta=\lambda_2+\frac{\lambda-\lambda_1}{n_0}$, for some $n_0\geq 2$ large enough, such that $\eta>\lambda_2$, Lemma 2.9 can also be proved. Then one can follow the proof of Theorem 1.1 and show that $\overline{\lambda}=\lambda_1$, i.e. obtaining the monotonicity up to maximal cap. Thus we deduce
\begin{equation*}
    u(x)=u(y_i-t \eta(x_i))\leq u(y_i-t _i\eta(x_i)),
\end{equation*}
with $y_i-t_i \eta(x_i)\in W_i\subset\omega_0$ which gives \cref{step1movingplane}.
\smallskip\\\textbf{Step $2$:} Once \cref{step1movingplane} is proved, we now show \cref{step2movingplane} in spirit of [\citealp{GidasSpruck}, Theorem 1.1]. We can treat the nonlocal part as a perturbation term and follow the blow up analysis. \smallskip\\We argue by contradiction. If possible, let there exist a sequence $\left\{u_k\right\}$ of positive solutions of \cref{approximated4}$_{k,\lambda}$ and a sequence of points $P_k \in\bar\omega_0\subset \Omega$ (as continuous functions on compact sets achieve maximum) such that
\begin{equation*}
    M_k=u_k(P_k)=\max \{u_k(x): x \in \bar{\omega}_0\} \rightarrow+\infty \text { as } k \rightarrow+\infty .
\end{equation*}
By Bolzano-Weirestrass theorem, up to a subsequence, we can assume that $P_k \rightarrow P \in \bar{\omega}_0$ as $k\rightarrow+\infty$. Let $2 d=\operatorname{dist}(\omega_0,\partial\Omega)$ and let $\omega_{0,d}$ be the set: $\omega_{0,d}=\{x \in \Omega: \operatorname{dist}(x, \omega_0) \leq d\}$. We take $\{\mu_k\}$ to be a sequence of positive numbers such that $\mu_k^{\frac{2}{r-1}} M_k=1$. Clearly $M_k \rightarrow+\infty$ implies $\mu_k \rightarrow 0$ as $k \rightarrow+\infty$. Now choose $\tilde{R}>0$ arbitrary and fix after choice. By convergence of $\{\mu_k\}$ we can select $k_0$ such that $B_{\tilde{R}+1}(0) \subset B_{\frac{d}{\mu_k}}(0)$ for all $k\geq k_0$. We now define the scaled function
\begin{equation*}
v_k(y)=\mu_k^{\frac{2}{r-1}} u_k(P_k+\mu_k y), \quad \forall y \in B_{\frac{d}{\mu_k}}(0).
\end{equation*}
By boundary estimate in Step 1, $v$ satisfies
\begin{equation*}
    \sup \{v_k(y): y \in B_{\frac{d}{\mu_k}}(0)\}=v_k(0)=1.
\end{equation*}
Further calculation yields that
\begin{equation*}
    \begin{split}
        -\Delta v_k(y)+\mu_k^{2-2s}(-\Delta)^s v_k(y)&=\mu_k^{\frac{2 r}{r-1}}\left(\lambda f_k(u_k(P_k+\mu_k y))+(\mu_k^{\frac{-2}{r-1}} v_k(y))^r\right), \quad y \in B_{\tilde{R}+1}(0), \\
v_k(0)&=1 .
 \end{split}
\end{equation*}
Now by \cref{subsuper}, every solution of \cref{approximated4} is a supersolution of \cref{approximated5} and we have $u_k \geq w_{k, \lambda}$, therefore it holds
\begin{equation*}
    \lambda f_k(u_k(P_k+\mu_k y)) \leq \frac{\lambda}{(w_{k,\lambda}(P_k+\mu_k y))^\gamma} .
\end{equation*}
Since $P_k+\mu_k y \in \bar{\omega}_{0,d} \subset \Omega$ for all $y \in B_{\bar{R}+1}(0)$, by \cref{greater}, $\exists \,C(\omega_{0})>0$ such that
\begin{equation*}
  \mu_k^{\frac{2 r}{r-1}}\left(\lambda f_k(u_k(P_k+\mu_k y))+(\mu_k^{\frac{-2}{r-1}} v_k(y))^r\right)\leq C(\omega_0), \quad y \in B_{\tilde{R}+1}(0), \quad \forall k \geq k_0 .
\end{equation*}
Therefore each $v_k$ satisfies 
\begin{equation*}
    \begin{split}
        -\Delta v_k+\mu_k^{2-2s}(-\Delta)^s v_k=h_k, \quad y \in B_{\tilde{R}+1}(0), \quad\text { and }\quad
v_k(0)=1 ,
 \end{split}
\end{equation*}
with $\|h_k\|_{L
^\infty(B_{\tilde{R}+1}(0))}\leq C$ for all $k$, where $C$ is independent of $k$. \smallskip\\Without loss of generality, we can assume $\mu_k<1$, and at this point, we can use the local H\"older bound estimates of \cite{MG}. In fact, it can be easily checked that the conditions $(1.5)-(1.7)$ of [\citealp{MG}, Theorem 3, Theorem 5] are satisfied and the boundary data $g_k(\cdot)=\mu_k^{\frac{2}{r-1}} u_k(P_k+\mu_k \cdot) $ satisfies the required condition. Further, a delicate observation of their proof indicates that the constants (which dominate the corresponding seminorms) are linearly dependent on $\|h_k\|_{L^n}$ (or $\|h_k\|_{L^d}, d>n$) and $\|v_k\|_{L^2}$. As $\|h_k\|_{L
^\infty}\leq C$ and $\|v_k\|_{L
^\infty}\leq 1$, therefore
we get $\{v_k\}$ is bounded in $C^{1, \beta}(B_{\tilde{R}}(0))$ for some $\beta\in (0,1)$. Now compact embedding of $C^{1, \beta}(B_{\tilde{R}}(0))$ in $C^{1, \alpha}(B_{\tilde{R}}(0))$ for $0<
\alpha<\beta$ implies there exists a subsequence of $\{v_{k}\}$ which converges in $C^{1, \alpha}\left(B_{\tilde{R}}(0)\right)$ to some $v$.
\\Now, we estimate by using uniform bound of H\"older seminorm of $v_k$ in $B_{\tilde{R}+1/2}(0)$ for every $\beta\in (0,1)$ [\citealp{MG}, Theorem 3]. As $s\in(0,1/2)$, we choose $\beta$ such that $\beta>2s$ and get
\begin{equation}{\label{unibddfrac}}
\begin{array}{c}|(-\Delta)^sv_k(x)|=\left|c_{n,s}\int_{\mathbb{R}^n}\frac{v_k(x)-v_k(y)}{|x-y|^{n+2s}} dy\right|\leq c_{n,s}\int_{ B_{1/2}(x)}\frac{|v_k(x)-v_k(y)|}{|x-y|^{n+2s}} dy+c_{n,s} \int_{\mathbb{R}^n\backslash B_{1/2}(x)}\frac{|v_k(x)-v_k(y)|}{|x-y|^{n+2s}} dy\smallskip\\\quad\quad\quad\quad\quad\quad\quad\quad\quad\quad\quad\quad\leq c_{n,s}\int_{ B_{1/2}(x)}\frac{[v_k]_{0,\beta,B_{\tilde{R}+1/2}(0)}}{|x-y|^{n+2s-\beta}} dy+ 2c_{n,s}\|v_k\|_{\infty}\int_{\mathbb{R}^n\backslash B_{1/2}(x)}\frac{1}{|x-y|^{n+2s}} dy \leq C,
\end{array}
\end{equation}
for all $x\in B_{\tilde{R}}(0)$ and for all $k$. Further, $t\to t^r$ is locally Lipschitz for $t> 0$ so for all $k$, as $\|v_k\|_{\infty}\leq 1$ it holds
\begin{equation}{\label{uniconv}}
    |(v_k(x))^r-(v(x))^r|\leq M|v_k(x)-v(x)|.
\end{equation}
Also 
\begin{equation}{\label{uniconv2}}
    \lambda f_k(u_k(P_k+\mu_k y))\leq C(\omega_0).
\end{equation}
As $v_k$ satisfies 
\begin{equation}{\label{blowup}}
    \begin{split}
        -\Delta v_k(y)+\mu_k^{2-2s}(-\Delta)^s v_k(y)&=\mu_k^{\frac{2 r}{r-1}}\left(\lambda f_k(u_k(P_k+\mu_k y))+(\mu_k^{\frac{-2}{r-1}} v_k(y))^r\right), \quad y \in B_{\tilde{R}}(0), 
\quad v_k(0)=1 ;
 \end{split}
\end{equation}
and $v_k\rightarrow v$ in $C^{1, \alpha}(B_{\tilde{R}}(0))$, so $v_k\to v$ uniformly and by \cref{unibddfrac}, \cref{uniconv,uniconv2}, we get taking limit in \cref{blowup},
\begin{equation*}
    \begin{split}
        -\Delta v= v^r \text{ in } B_{\tilde{R}}(0), \quad v(0)=1.
 \end{split}
\end{equation*}
Now, taking larger and larger balls, we obtain a Cantor diagonal subsequence, which converges to
$v \in C^{1,\alpha}(\mathbb{R}^n )$ on all compact subsets of $\mathbb{R}^n$ and $v$ satisfies
\begin{equation*}
    \begin{split}
        -\Delta v= v^r \text{ in } \mathbb{R}^n, \quad v(0)=1,
 \end{split}
\end{equation*}
with $1<r<\frac{n+2}{n-2}$. By [\citealp{GidasSpruckBlowup}, Theorem 1.1] this a contradiction proving Step 2 and thus the lemma concludes.
\end{proof}
Before proceeding, we prove [\citealp{gamez}, Theorem 2.2] for our context. Let $C^m_0(\overline{\Omega})$ be the space of $m$ times differentiable functions vanishing in $\mathbb{R}^n\backslash\Omega$. 
Let us define the set
\begin{equation*}
    \mathbb{P}=\left\{u \in C_0^{1}(\overline{\Omega}): u(x) \geq 0 \text { in } \overline{\Omega}\right\} .
\end{equation*}
Clearly, the interior of $ \mathbb{P}$ is
\begin{equation*}
\mathbb{P}^{\sim}=\left\{u \in C^{1}(\overline{\Omega}): u>0 \text { in } \Omega \text { and } \frac{\partial u}{\partial \eta}(x)<0 \text { for all } x \in \partial \Omega\right\},
\end{equation*}
 where $\eta$ is the unit outward normal to $\partial \Omega$.
\begin{lemma}{\label{prelim}}
    Suppose $u$ and $\bar{u}$ are the solution and super-solution to \cref{approximated4} in $C_0^{2}(\overline{\Omega})$ and $\bar u-u\in \mathbb{P}$. If $u \neq \bar{u}$, then $\bar{u}-u$ is not in $\partial \mathbb{P}$, where $\partial \mathbb{P}$ is the boundary of $\mathbb{P}$.
\end{lemma}
\begin{proof} Suppose by contradiction that $\bar{u}-u \in \partial \mathbb{P}$. Hence, we have $\bar{u}(x) \geq u(x)$ for all  $x\in\overline\Omega$. Setting $z=\bar u-u$, if $M$ is the Lipschitz constant of $\lambda f_k+g$ in $[\underset{\bar{\Omega}}{\min} \, u, \underset{\bar{\Omega}}{\max} \, \bar u]$, we have 
\begin{equation*}
  -  \Delta z(x)+(-\Delta)^s z(x)+ M z(x)\geq \lambda f_k(\bar u(x))+g(\bar u(x))-\lambda f_k(u(x))-g(u(x))+ M z(x)\geq 0, \quad \text{ pointwise in } \Omega. 
\end{equation*}
Proceeding now as \cref{subsuper}, we arrive at $\bar{u}-u \in \mathbb{P}^{\sim}$, which contradicts our assumption as $\mathbb{P}^{\sim} \cap \partial \mathbb{P}=\emptyset$.
\end{proof}
\begin{lemma}{\label{important}}
Let $I \subset \mathbb{R}$ be an interval and $\Sigma \subset I \times C_0^{2}(\overline{\Omega})$ be a connected set of solutions to \cref{approximated4}. Consider a continuous map $U:I \rightarrow C_0^{2}(\overline{\Omega})$ such that $U(\lambda)$ is a super-solution of \cref{approximated4}$_\lambda$ for every $\lambda \in I$, but not a solution. If $u_0 \leq U(\lambda_0)$ in $\Omega$ but $u_0 \neq U(\lambda_0)$ for some $(\lambda_0, u_0) \in \Sigma$ then $u<U(\lambda)$ in $\Omega$ for all $(\lambda, u) \in \Sigma$.
\end{lemma}
\begin{proof} Consider the continuous map
$
T: I \times C_0^{1}(\overline{\Omega}) \rightarrow C_0^{1}(\overline{\Omega}) \text { given by } T(\lambda, u)=U(\lambda)-u \text {. }
$
By continuity of $T$ it holds $T(\Sigma)$ is connected in $C_0^{2}(\overline{\Omega})$. Now by \cref{prelim}, $ T(\Sigma)$ completely lies in $\mathbb{P}^{\sim}$ or completely outside $\mathbb{P}$. Since $T(\lambda_0,u_0)\in \mathbb{P}^\sim$, we conclude $T(\Sigma) \subset \mathbb{P}^{\sim}$ and, therefore, $u<U(\lambda)$ for all $(\lambda, u) \in \Sigma$.
\end{proof}
We now show the inverse of $-\Delta +(-\Delta)^s$ is a compact operator. For convenience, we write $\mathcal{L}=-\Delta+(-\Delta)^s$.
As for any $w\in L^2(\Omega)$, the Dirichlet problem
\begin{equation*}
\mathcal{L} u= w \text{ in }\Omega, \quad u=0 \text{ in } \mathbb{R}^n\backslash\Omega;
\end{equation*}
has a unique weak solution (see \cite{Biagiregularitymaximum,FaberKrahn}) satisfying $||u||_{W^{1,2}_0(\Omega)}\leq c||w||_{L^2(\Omega)}$, where $c>0$ is a constant independent of $w$; and by regularity up to the boundary we have (see for example \cite{FaberKrahn,MG,valdinoci}) $u\in C(\overline\Omega)$ for all $w\in C(\overline\Omega)$, therefore the map $K=\mathcal{L}^{-1}$ is well defined from $C(\overline\Omega)$ to $C(\overline\Omega)$.
\begin{lemma}{\label{compactopt}}
The operator $\mathcal{L}^{-1}$ i.e. $K:C(\overline\Omega)\to C(\overline\Omega)$ is compact.
\end{lemma}
\begin{proof}
\textbf{Step $1$:} First we show for any $w\in L^\infty(\Omega)%C(\overline\Omega)
    $, the $L^\infty$ norm of $K(w)=u$ is linearly dominated by $||w||_{L^\infty(\Omega)}$. Following [\citealp{kinderlehrer2000introduction}, Theorem B.2], we define
  %  \begin{equation*}
$A(k):=\{x \in \Omega ;| u(x)| \geq k\} \quad \text { for any } k > 0 $. %\end{equation*}
 Choosing $\phi_k:=(\operatorname{sgn} u)(|u|-k)^{+}\in W^{1,2}_0(\Omega)$ as a test function, we get
\begin{equation*}
\int_{\Omega}\nabla u \cdot\nabla \phi_k\, d x+\int_{\mathbb{R}^n} \int_{\mathbb{R}^n} \frac{(u(x)-u(y))(\phi_k(x)-\phi_k(y))}{|x-y|^{n+2s}} d x dy=\int_{\Omega} w \phi_k\, d x .
\end{equation*}
It is easy to observe that the nonlocal integral is nonnegative
and hence the continuity of the mapping $W_0^{1, 2}(\Omega) \hookrightarrow L^l(\Omega)$ for some $2<l\leq 2^*$ gives
\begin{equation*}
    \begin{array}{rcl}
    \int_{\Omega}|\nabla \phi_k|^2 d x=\int_{\Omega}\nabla u \cdot \nabla \phi_k\, d x &\leq& \int_{\Omega} w \phi_k \,d x \leq\|w\|_{L^{\infty}(\Omega)} \int_{A(k)}\phi _k\,d x \\
&\leq& C_0\|w\|_{L^{\infty}(\Omega)}|A(k)|^{\frac{l-1}{l}}\left(\int_{\Omega}\left|\nabla \phi_k\right|^2d x\right)^{\frac{1}{2}},
    \end{array}
\end{equation*}
where $C_0$ is the Sobolev constant. Hence, we have
\begin{equation*}
    \int_{\Omega}|\nabla \phi_k|^2 d x \leq C\|w\|^2_{L^{\infty}(\Omega)}|A(k)|^{\frac{2(l-1)}{l}}.
\end{equation*}
Now choose $h>k$. Clearly $A(h) \subset A(k)$. Using this fact and the above inequality, we estimate as
\begin{equation*}
\begin{array}{rcl}(h-k)^2|A(h)|^{\frac{2}{l}}\leq\left(\int_{A(h)}(|u(x)|-k)^l d x\right)^{\frac{2}{l}} &\leq&\left(\int_{A(k)}(|u(x)|-k)^l d x\right)^{\frac{2}{l}} \smallskip\\
&\leq &C_0 \int_{\Omega}\left|\nabla \phi_k\right|^2 d x \leq  C\|w\|^2_{L^{\infty}(\Omega)}|A(k)|^{\frac{2(l-1)}{l}}.
\end{array}
\end{equation*}
 Therefore, we have
$
|A(h)| \leq C\frac{\|w\|^l_{L^{\infty}(\Omega)}}{(h-k)^l}|A(k)|^{l-1} ,
$ for all $h>k>0$.
Thus using [\citealp{kinderlehrer2000introduction}, Lemma B.1] we obtain
$
   |A(d )|=0$, where $d^l=c\|w\|^l_{L^{\infty}(\Omega)}$, for a constant $c$ depending on $l, \Omega, C_0$ only. Therefore $\|u\|_{L^\infty(\Omega)}\leq c\|w\|_{L^\infty(\Omega)}.$\smallskip\\
\textbf{Step $2$:} Now let $\{w_k\}$ be a bounded sequence in $C(\overline{\Omega}).$ By [\citealp{valdinoci}, Theorem 1.4], we then have for each $u_k=K(w_k)$,
\begin{equation*}
    \|u_k\|_{W^{2,p}(\Omega)}\leq C(n,s,p)(\|u_k\|_{L^p(\Omega)}+\|w_k\|_{L^p(\Omega)})\leq C(\|u_k\|_{L^\infty(\Omega)}+\|w_k\|_{L^\infty(\Omega)})\leq C\|w_k\|_{L^\infty(\Omega)}\leq C,
\end{equation*}
for all $1<p<\infty$ and some constant $C\equiv C(n,p,s,\Omega)$ independent of $k$. Note that we have used the result of Step $1$ in the second last inequality. Now by compact embedding of $W^{2,p}(\Omega)$ in $C^{1,\alpha}(\overline\Omega)$, for $p>n$ we get $\{u_k\}$ has a convergent subsequence in $C(\overline\Omega)$, concluding our proof.
\end{proof}
With these preliminary results, we now show \cref{approximated4}$_k$ has at least two distinct solutions for all $\lambda\in(0,\Lambda)$, $\Lambda$ to be determined later. We only restrict our case to $\gamma\geq 1$, as one can find multiplicity result for $\gamma<1$ in \cite{garaingeometric}. 
\begin{lemma}{\label{twosol}}
    Let $\gamma\geq 1$. Then there exists $\bar k\in\mathbb{N}$ and $\lambda >0$ such that for any
$k \geq \bar k$, problem \cref{approximated4} admits at least two distinct solutions $u_k, v_k \in W^{1,2}_0(\Omega)\cap L^\infty(\Omega)$, provided $0 < \lambda< \Lambda$.
\end{lemma}
\begin{proof}
    We prove the result in several steps.
\\
\textbf{Step $1$:} (Existence of a super-solution which is not a solution) Taking $T$ and $\delta_0$ to be same as \cref{boundednessofsubsolution} and \cref{delta}, we fix $\delta_1:=(2 r-1)^{\frac{1}{1-r}} T^{\frac{\gamma+1}{1-r}}$, and take $\delta_2 \in(0, \min \{\delta_0, \delta_1\})$ to define:
\begin{equation*}
\Lambda:=\max _{0 \leq t \leq \delta_2} q(t), \quad \text { where } q(t)=\frac{1}{2}\left(\left(\frac{t}{T}\right)^{\gamma+1}-t^{\gamma+r}\right) .    
\end{equation*}
Fix any $\lambda_0 \in(0, \Lambda)$. As the function $q$ is strictly positive in $(0, \delta_2]$ (and so $\Lambda>0$), by the IVP of continuous functions, there exists $\delta \in(0, \delta_2]$ such that
$
\lambda_0=q(\delta) .
$ Moreover $\delta\leq \delta_2<\delta_1$ implies
\begin{equation}{\label{firstexi}}
\frac{r-1}{\gamma+1} \delta^{r+\gamma}<(r-1) \delta^{r+\gamma}<\lambda_0 .    
\end{equation}
Setting $\lambda^*:=(\delta / T)^{\gamma+1}$ we have
\begin{equation*}
\lambda^*>\lambda_0+\left(T(\lambda^*)^{\frac{1}{\gamma+1}}\right)^r\left(T(\lambda^*)^{\frac{1}{\gamma+1}}\right)^\gamma,
\end{equation*}
and this allows us to choose $k_0 \in \mathbb{N}$ such that
\begin{equation*}
\lambda^* \geq \lambda+\left(T(\lambda^*)^{\frac{1}{\gamma+1}}\right)^r\left(T(\lambda^*)^{\frac{1}{\gamma+1}}+\frac{1}{k}\right)^\gamma, \quad \forall k \geq k_0, \forall \lambda \in\left[0, \lambda_0\right] .
\end{equation*}
Now let $w_{k, \lambda^*} \in C_0^2(\overline{\Omega})$ (see \cref{contin}) be solution to \cref{approximated5} with $\lambda=\lambda^*$. By \cref{boundednessofsubsolution}, $w_{k,\lambda^*}$ satisfies $\left\|w_{k, \lambda^*}\right\|_{L^\infty(\Omega)} \leq T(\lambda^*)^{\frac{1}{\gamma+1}}=\delta$. Therefore,
\begin{equation*}
\lambda^* \geq \lambda+\left\|w_{k, \lambda^*}\right\|_{\infty}^r\left(\left\|w_{k, \lambda^*}\right\|_{\infty}+\frac{1}{k}\right)^\gamma \geq \lambda+(w_{k, \lambda^*})^r\left(w_{k, \lambda^*}+\frac{1}{k}\right)^\gamma
\end{equation*}
from which it holds
\begin{equation*}
    -\Delta w_{k, \lambda^*}+(-\Delta)^s w_{k,\lambda^*}=\frac{\lambda^*}{(w_{k, \lambda^*}+\frac{1}{k})^\gamma} \geq \frac{\lambda}{(w_{k, \lambda^*}+\frac{1}{k})^\gamma}+\left(w_{k, \lambda^*}\right)^r, \quad \forall k\geq k_0, \forall \lambda \in\left[0, \lambda_0\right],
\end{equation*}
thereby implying $w_{k, \lambda^*} \in C_0^2(\overline{\Omega})$ is a supersolution (and not a solution) to \cref{approximated4} for all $k \geq k_0$ and all $\lambda \in\left[0, \lambda_0\right]$ with $ \left\|w_{k, \lambda^*}\right\|_{L^\infty(\Omega)} \leq \delta$.   \smallskip\\
\textbf{Step $2$:} (Existence of a unique solution with a particular small norm) The function $(r-1) t^r\left(t+\frac{1}{k}\right)^{\gamma+1}$ is convex, increasing, starts from $0$ and hence intersects any straight line (not passing through origin) with positive slope at a unique point. So $\exists! \,M_k=M_k(\lambda)>0$, increasing with respect to the parameter $\lambda$, such that
\begin{equation*}
\begin{array}{c}
    (r-1) M_k^r\left(M_k+\frac{1}{k}\right)^{\gamma+1}=\lambda\left(M_k(\gamma+1)+\frac{1}{k}\right) \text{and }
    (r-1) t^r\left(t+\frac{1}{k}\right)^{\gamma+1} <\lambda\left(t(\gamma+1)+\frac{1}{k}\right),  \forall t \in[0, M_k).
\end{array}
\end{equation*}
Using this, one can derive (as the derivative is negative) that $\frac{\lambda f_k(t)+g(t)}{t}$ is decreasing in $[0, M_k)$ and thus, by [\citealp{eigenvalue}, Theorem 4.3], there exists at most one solution $u_k$ to \cref{approximated4} with $\left\|u_k\right\|_{\infty}<M_k$. Now by \cref{firstexi}, $\exists \ep>0$ such that $\frac{(r-1)}{\gamma+1}(\delta+\epsilon)^{r+\gamma}%<(r-1)(\delta+\epsilon)^{r+\gamma}
<\lambda_0$. This implies $\exists k_1 \in \mathbb{N}$ such that
$
\lambda_k:=\frac{(r-1)(\delta+\epsilon)^r\left((\delta+\epsilon)+\frac{1}{k}\right)^{\gamma+1}}{(\delta+\epsilon)(\gamma+1)+\frac{1}{k}}<\lambda_0 $ for all $ k \geq k_1.
$ By monotonicity of $M_k$ with respect to $\lambda$, we conclude
$
M_k(\lambda_0) \geq M_k(\lambda_k)=\delta+\epsilon,\forall k \geq k_1 .
$
\\
\textbf{Step $3$:} (Nonexistence of solution for large $\lambda$) By \cref{property1}, $\lim _{t\rightarrow 0^{+}} \frac{\lambda_1 t-g(t)}{f_1(t)}=0$, while by \cref{property2}, $\lim _{t\rightarrow+\infty} \frac{\lambda_1 t-g(t)}{f_1(t)}=-\infty$, where $\lambda_1$ is the first eigenvalue corresponding to $-\Delta+(-\Delta)^s$. We fix $\bar{\Lambda}:=\max _{t>0} \frac{\lambda_1 t-g(t)}{f_1(t)}>0$. Let $\exists$ a positive solution $u \in W_0^{1,2}(\Omega)$ to \cref{approximated4} for $\lambda>\bar\Lambda$. Taking $e_1>0$ (the first eigenvector) as test function, we get
\begin{equation*}
    \int_{\Omega}(\lambda_1 u-\lambda f_k(u)-g(u)) e_1=0 .
\end{equation*}
As $f_1(t) \leq f_k(t)$ for all $t>0$, from this last inequality we deduce
\begin{equation*}
\lambda<\max _{t>0} \frac{\lambda_1 t-g(t)}{f_k(t)} \leq \bar{\Lambda},
\end{equation*}
leading to a contradiction, and the step is proved.
\smallskip\\\textbf{Step $4$:} (Existence of two distinct solutions) By \cref{contin}, it is enough to show existence of solutions in $C(\overline\Omega)$. We fix $k\geq \max \left\{k_0, k_1\right\}=\bar k$, where $k_0$ and $k_1$ are given by Step 1 and Step 2, respectively. Next we define the operator $K_\lambda: C(\overline{\Omega}) \longrightarrow C(\overline{\Omega})$ by
\begin{equation*}
K_\lambda(u)=[-\Delta+(-\Delta)^s]^{-1}\left(\lambda f_k(u)+g(u)\right), \quad u \in C(\overline{\Omega}), \quad u\geq 0 \text{ in }\Omega .
\end{equation*}
From \cref{compactopt}, we have $K_\lambda$ is a compact operator for every $\lambda$. Further, solutions of \cref{approximated4} are fixed points of $K_\lambda$. As by Step 3, \cref{approximated4} does not admit any solution for $\lambda>\bar\Lambda$, so for $0\leq\lambda<\bar\Lambda$, choose $R_k$ (that depends on $k$) such that every solution $u$ of \cref{approximated4}$_k$ satisfies $\|u\|_{\infty}<R_k$ (indeed the same process of \cref{uniformboundedness} can be followed to get this). Consider the positive cone of $C(\overline{\Omega})$ given by
\begin{equation*}
    C=\left\{u \in C(\overline{\Omega}): u\geq 0 \text { in } \Omega\right\}, R:=R_k,
\end{equation*}
And define the compact map
$
    K_0: C \rightarrow C \text { by } K_0(u)=[-\Delta+(-\Delta)^s]^{-1} u^r.
$
Note that, we can consider the Leray-Schauder topological degree of $I-K_0$, i.e. $d(I-K_0, B_{R}, 0)$. Further, since $\delta<\delta_0$ ($\delta$ is such that $\lambda_0=q(\delta)$, by step 1), by \cref{delta}, the problem \cref{approximated4}$_{\lambda=0}$ has no solution on the boundary of the ball $B_\delta$. Consequently, one can also consider the Leray-Schauder topological degree $d\left(I-K_0, B_\delta, 0\right)$. Applying [\citealp{nussbaum}, Proposition 2.1], we get $d(I-K_0, B_{R}, 0)=0$ and $d(I-K_0, B_\delta, 0)=1$. Now setting $X=C(\overline\Omega), a=0, b=\bar\Lambda, U=B_{R}, U_1=B_\delta, T(\lambda,u)=K_\lambda(u)$ in [\citealp{implemma}, Theorem 4.4.2] we get the existence of a continuum (connected and closed) $S_k \subset \Sigma_k=\{(\lambda, u_k) \in[0,+\infty) \times C(\overline{\Omega}): u_k \text{ is solution of \cref{approximated4}}\}$ such that
\begin{equation}{\label{sec}}
    (0,0) \in S_k\text { and } S_k \cap\left(\{0\} \times(C(\overline{\Omega}) \backslash B_\delta)\right) \neq \varnothing, \quad \forall k \in \mathbb{N} .
\end{equation}
At this point, we define the continuous map $U:[0, \lambda_0] \longrightarrow C_0^2(\overline {\Omega})$ by $U(\lambda)=w_{k, \lambda^{*}}$, for every $\lambda \in\left[0, \lambda_0\right]$, then by Step 1, $U(\lambda)$ is a positive supersolution and not a solution to \cref{approximated4} for all $\lambda \in\left[0, \lambda_0\right]$. Since $\Omega$ satisfies the interior sphere condition, we apply \cref{important} to deduce that every pair $(\lambda, u_k)$ belonging to the connected component of $S_k \cap\left(([0, \lambda_0] \times C(\overline{\Omega})\right)$ which emanates from $(0,0)$ lies pointwise below the branch $\left\{(\lambda, {U}(\lambda)) : 0 \leq \lambda \leq \lambda_0\right\}$ at least until it crosses $\lambda=\lambda_0$. In particular, there exists $u_k$ in the slice $S_k^{\lambda_0}=\left\{u \in C(\overline{\Omega}):(\lambda_0, u) \in S_k\right\}$ and satisfies $0<u_k<w_{k, \lambda^*}$. Recalling that $\left\|w_{k, \lambda^*}\right\| \leq \delta$, we have $\left\|u_k\right\|_{\infty} \leq\left\|w_{k, \lambda^*}\right\| \leq \delta$. By Step 2, it is clear that $u_k$ is the unique solution of \cref{approximated4} with norm less than or equal to $\delta+\epsilon$. Again by \cref{sec} $ S_k \cap\left(\{0\} \times(C(\overline{\Omega}) \backslash B_{\delta+\varepsilon})\right) \neq \varnothing$ and so we conclude also the existence of $v_k$ (solution to \cref{approximated4}) in $S_k^{\lambda_0}$ with $\left\|v_k\right\|_{\infty} \geq \delta+\epsilon$. Clearly, $u_k$ and $v_k$ are different. Hence, we have found the existence of two distinct solutions for $\lambda = \lambda_0$ and since $\lambda_0<\bar\Lambda$
is arbitrary, we have the required result. 
\end{proof}
\section{Proof of \cref{mainth3}}{\label{th3}
\textbf{Step $1$:} We will deal with the first nonnegative eigenfunction $\phi_1$ (corresponding to the eigenvalue $\tilde\lambda_1$) of the Laplace operator. As $\Omega$ has smooth boundary, so $\phi_1$ is smooth up to the boundary. We extend $\phi_1\equiv 0$ outside $\Omega$. Since $\phi_1\in C^{\infty}(\overline\Omega)$ i.e. in particular continuously differentiable in $\overline\Omega,$ therefore $\phi_1$ is Lipschitz in $\overline\Omega$. Moreover $\phi_1\equiv 0$ in $\mathbb{R}^n\backslash \Omega$ implies $\phi_1\in C^{0,1}(\mathbb{R}^n)$. We now show that $-\Delta z(x)+(-\Delta)^sz(x)\leq \frac{C_0}{(w(x))^{\frac{2\gamma}{(\gamma+1)}}}M$, where \begin{equation*}z(x)=(w(x))^{\frac{2}{(\gamma+1)}}-\frac{1}{k}, \quad w(x)=\left(C_0\phi_1(x)+\frac{1}{k^{\frac{(\gamma+1)}{2}}}\right)(>0),\end{equation*}$M>0$ is a constant independent of $x$ and $C_0$ is a constant to be determined later. For $x\in\Omega$ estimating the Laplacian, we get
\begin{equation}{\label{lapla}}
    -\Delta z(x)=\frac{C_0}{w(x)^{2\gamma/(\gamma+1)}}\left\{\frac{2C_0(\gamma-1)}{(\gamma+1)^2}|\nabla\phi_1|^2+\frac{2\tilde\lambda_1\phi_1(x)}{\gamma+1}w(x)\right\}.
\end{equation}
Further, for $x\in\Omega$, we estimate the fractional Laplacian as
\begin{equation*}{\label{fraclap}}
    \begin{array}{rcl}
    (-\Delta)^sz(x)&=&c_{n,s}\int_{\mathbb{R}^n}\frac{w(x)^{2/(\gamma+1)}-w(y)^{2/(\gamma+1)}}{|x-y|^{n+2s}} dy\smallskip\\&=&\frac{c_{n,s}}{w(x)^{2\gamma/(\gamma+1)}}\int_{\mathbb{R}^n}\frac{w(x)^2-w(x)^{2\gamma/(\gamma+1)}w(y)^{2/(\gamma+1)}}{|x-y|^{n+2s}} dy\smallskip\\&\leq&
\frac{c_{n,s}}{w(x)^{2\gamma/(\gamma+1)}}\int_{\mathbb{R}^n}\frac{|w(x)^2-w(y)^2|}{|x-y|^{n+2s}} dy  \smallskip\\&\leq & \frac{2c_{n,s}\|w\|_{\infty}}{w(x)^{2\gamma/(\gamma+1)}}\int_{\mathbb{R}^n}\frac{|w(x)-w(y)|}{|x-y|^{n+2s}} dy \leq\frac{2C_0c_{n,s}(C_0\|\phi_1\|_{\infty}+1)}{w(x)^{2\gamma/(\gamma+1)}}\int_{\mathbb{R}^n}\frac{|\phi_1(x)-\phi_1(y)|}{|x-y|^{n+2s}}  dy.\end{array}
\end{equation*}
Now by Lipschitz continuity of $\phi_1$ and using $s\in(0,1/2)$, we get
\begin{equation}{\label{fraclapla}}
    \begin{array}{rcl}
       (-\Delta)^sz(x)&\leq &   \frac{2C_0c_{n,s}(C_0\|\phi_1\|_{\infty}+1)}{w(x)^{2\gamma/(\gamma+1)}}\left[\int_{B_1(x)}\frac{|\phi_1(x)-\phi_1(y)|}{|x-y|^{n+2s}} dy  +\int_{\mathbb{R}^n\backslash B_1(x)}\frac{|\phi_1(x)-\phi_1(y)|}{|x-y|^{n+2s}} dy \right]\smallskip\\&\leq&  \frac{2C_0c_{n,s}(C_0\|\phi_1\|_{\infty}+1)}{w(x)^{2\gamma/(\gamma+1)}}\left[\int_{B_1(x)}\frac{[\phi_1]_{C^{0,1}(\mathbb{R}^n)}}{|x-y|^{n+2s-1}} dy  +2\|\phi_1\|_\infty\int_{\mathbb{R}^n\backslash B_1(x)}\frac{1}{|x-y|^{n+2s}}  dy \right].%\smallskip\\&\leq&\frac{C_0C}{w(x)^{2\gamma/(\gamma+1)}},
    \end{array}
\end{equation}
We conclude our claim by clubbing \cref{lapla,fraclapla} and using the boundedness of $\phi_1, \nabla\phi_1$. \smallskip\\\textbf{Step $2$:} Fix $\lambda_0\in (0,\Lambda)$ (see Step 1 of \cref{twosol}). We already have by \cref{twosol}, the existence of two different sequence of solutions for \cref{approximated4}. It suffices to show that they converge to two different solutions of \cref{prob2}.\\ By Step 1, we can choose $C_0$ (depending on $\lambda_0$) so small that $-\Delta z+(-\Delta)^sz \leq \frac{\lambda_0}{(z+\frac{1}{k})^\gamma}$. Therefore $z$ is a subsolution to \cref{approximated5} for $\lambda=\lambda_0.$ Moreover, since $\frac{\lambda_0}{(t+\frac{1}{k})^\gamma} \leq \frac{\lambda_0}{(t+\frac{1}{k})^\gamma}+t^r$ for any $t \geq 0$, each solution $u$ of \cref{approximated4}$_{\lambda_0}$ is a supersolution of \cref{approximated5}$_{\lambda_0}$. By \cref{subsuper} we have $z\leq w_{k, \lambda_0} \leq u$. In particular, the sequences $u_k$ and $v_k$ obtained in \cref{twosol} satisfy:
\begin{equation}{\label{unibddtwosol}}
\begin{array}{c}
 \left(C_0 \phi_1+\frac{1}{k^{(\gamma+1) / 2}}\right)^{2 /(\gamma+1)}-\frac{1}{k} \leq w_{k, \lambda_0} \leq u_k \leq \delta, \\
 \left(C_0 \phi_1+\frac{1}{k^{(\gamma+1) / 2}}\right)^{2 /(\gamma+1)}-\frac{1}{k} \leq w_{k, \lambda_0} \leq v_k, \quad\left\|v_k\right\|_{\infty} \geq \delta+\varepsilon>\delta .
\end{array}
\end{equation}
We write $z_k=u_k$ or $v_k$ and note that $z_k\in W^{1,2}_0(\Omega)$. Then by \cref{unibddtwosol} and \cref{uniformboundedness}, we have
\begin{equation}{\label{bddsol}}
    \left(C \phi_1+\frac{1}{k^{(\gamma+1) / 2}}\right)^{2 /(\gamma+1)}-\frac{1}{k} \leq w_{k,\lambda_0}\leq z_k \leq M,
\end{equation}
where $M$ is a positive constant independent of $k$. Further, by \cref{greater}, for every $\omega \subset \subset \Omega$ there exists $c(\omega)$ (recall that $\lambda_0$ is fixed) such that
$
    z_k(x) \geq c(\omega)>0, \text { for every } x \in \omega, \text { for every } k \in \mathbb{N} .
$
One can now follow \cref{mainth2} and show $z_k$ is bounded in $W_{\operatorname{loc}}^{1,2}(\Omega)$, and the pointwise (a.e.) limit $u$ (of $\{u_k\}$) and $v$ (of $\{v_k\}$) are solutions to \cref{prob2}. As $\left\|u_k\right\|_{\infty} \leq \delta$ and $\left\|v_k\right\|_{\infty} \geq \delta+\epsilon>\delta$ we deduce that $u$ and $v$ are different.\\ \textbf{Step $3$:}
Fixing $\alpha>\frac{\gamma+1}{4}$ and $\theta=2 \alpha-1>\frac{\gamma-1}{2}$, we take $\phi=\left(z_k+\frac{1}{k}\right)^\theta-\left(\frac{1}{k}\right)^\theta$ as a test function in \cref{approximated4}$_{\lambda_0}$ to obtain
\begin{equation}{\label{1}}
    \begin{array}{l}
    \int_\Omega\nabla z_k\cdot \nabla\left(z_k+\frac{1}{k}\right)^\theta dx +\int _{\mathbb{R}^n}\int _{\mathbb{R}^n}\frac{(z_k(x)-z_k(y))\left((z_k(x)+\frac{1}{k})^\theta-(z_k(y)+\frac{1}{k})^\theta\right)}{|x-y|^{n+2s}}dx dy\smallskip\\=\lambda_0 \int_{\Omega} \frac{(z_k+\frac{1}{k})^\theta-(\frac{1}{k})^\theta}{(z_k+\frac{1}{k})^\gamma}dx+\int_{\Omega}\left(\left(z_k+\frac{1}{k}\right)^\theta-\left(\frac{1}{k}\right)^\theta\right) z_k^r\,dx.
    \end{array}
\end{equation}
Now using item $(i)$ of \cref{algebraic}, we get
\begin{equation}{\label{2}}
    \begin{array}{l}
    \frac{4 \theta}{(\theta+1)^2} \int _{\mathbb{R}^n}\int _{\mathbb{R}^n}\frac{\left(\left((z_k(x)+\frac{1}{k})^{\frac{\theta+1}{2}}-(\frac{1}{k})^\theta\right)-\left((z_k(y)+\frac{1}{k})^{\frac{\theta+1}{2}})-(\frac{1}{k})^\theta\right)\right)^2}{|x-y|^{n+2s}}dx dy\smallskip\\\leq    \int _{\mathbb{R}^n}\int _{\mathbb{R}^n}\frac{(z_k(x)-z_k(y))\left((z_k(x)+\frac{1}{k})^\theta-(z_k(y)+\frac{1}{k})^\theta\right)}{|x-y|^{n+2s}}dx dy.
    \end{array}
\end{equation}
By \cref{2}, the nonlocal integral of \cref{1} is nonengative. Therefore we get (also note \cref{embedding2})% (note that \cref{embedding2} gives the embedding of fractional Sobolev space into $W^{1,2}_0(\Omega)$) and we get 
\begin{equation*}
\begin{array}{rcl}
\frac{4 \theta}{(\theta+1)^2} \int_{\Omega}\left|\nabla\left(\left(z_k+\frac{1}{k}\right)^\alpha-\left(\frac{1}{k}\right)^\alpha\right)\right|^2 & =&\theta \int_{\Omega}\left(z_k+\frac{1}{k}\right)^{\theta-1}\left|\nabla z_k\right|^2 dx\smallskip\\
& \leq&\lambda_0 \int_{\Omega} \frac{\left(z_k+\frac{1}{k}\right)^\theta-\left(\frac{1}{k}\right)^\theta}{\left(z_k+\frac{1}{k}\right)^\gamma}dx+\int_{\Omega}\left(\left(z_k+\frac{1}{k}\right)^\theta -\left(\frac{1}{k}\right)^\theta\right) z_k^r\,dx \smallskip\\
& \leq &\lambda_0 \int_{\Omega}\left(z_k+\frac{1}{k}\right)^{\theta-\gamma}dx+\int_{\Omega}\left(z_k+\frac{1}{k}\right)^\theta z_k^r \,dx \smallskip\\
& \leq &\lambda_0 \int_{\Omega}\left(z_k+\frac{1}{k}\right)^{\theta-\gamma}dx+\int_{\Omega}\left(z_k+1\right)^\theta z_k^r \,dx.
\end{array}
\end{equation*}
By \cref{bddsol}, this implies $\left\{\left(z_k+\frac{1}{k}\right)^\alpha-\left(\frac{1}{k}\right)^\alpha\right\}$ is bounded in $W_0^{1,2}(\Omega)$ for $\theta \geq \gamma$. While for $\theta<\gamma$, we observe
\begin{equation*}
\left(z_k+\frac{1}{k}\right)^{\theta-\gamma} \leq\left(C_0 \phi_1+\frac{1}{k^{(\gamma+1) / 2}}\right)^{\frac{2(\theta-\gamma)}{1+\gamma}} \leq\left(C _0\phi_1\right)^{\frac{2(\theta-\gamma)}{1+\gamma}}.
\end{equation*}
Since $\theta>\frac{\gamma-1}{2}$, hence $\int_\Omega \phi_1^{\frac{2(\theta-\gamma)}{1+\gamma}}dx <+\infty$ (see \cite{mohammed}). Therefore $\left\{\left(z_k+\frac{1}{k}\right)^\alpha-\left(\frac{1}{k}\right)^\alpha\right\}$ is bounded in $W_0^{1,2}(\Omega)$. Consequently, a subsequence of $\left\{\left(z_k+\frac{1}{k}\right)^\alpha-\left(\frac{1}{k}\right)^\alpha\right\}$ is weakly convergent in $W_0^{1,2}(\Omega)$. Finally by pointwise convergence of $\{z_k\}$, we have the weak limit is $u^\alpha $ or $v^\alpha $ and hence $u^\alpha, v^\alpha\in W^{1,2}_0(\Omega).$
\begin{remark}
      As mentioned before, for $\gamma<3$, the above theorem shows that the solutions are in $W_0^{1,2}(\Omega)$ and, in particular, belong to $W_0^{1, q}(\Omega)$ for every $q<2$. Indeed if $\gamma\geq 3$, one can get that the solutions of \cref{prob2} still belong to $W_0^{1, q}(\Omega)$ for some $q<2$, we refer [\citealp{Arcoyamultigreaterthan}, Remark 1] for this.
\end{remark}
\section*{A related problem} In spirit of [\citealp{Arcoyamultigreaterthan}, Theorem 3], we mention that the same result can be obtained for mixed operators. We briefly outline the result and the proof. Consider the problem 
\begin{equation}{\label{prob3}}
\begin{array}{c}
-\Delta u+(-\Delta)^s u=\frac{f}{u^{\gamma}}\text { in } \Omega, \smallskip\\u>0 \text{ in } \Omega,\quad\quad\quad u=0  \text { in }\mathbb{R}^n \backslash \Omega;
\end{array}
\end{equation}
where $f\in L^m(\Omega), m>1$. By a solution of \cref{prob3}, we mean a function $u\in W^{1,2}_{\operatorname{loc}}(\Omega)$ such that $u>c_\omega>0$ (where $c_\omega$ is a constant), in $\omega$ for every $\omega\subset\subset\Omega$, and satisfies
\begin{equation*}
\int_{\Omega} \nabla u \cdot\nabla \phi\, d x+\int_{\mathbb{R}^n}\int_{\mathbb{R}^n}\frac{(u(x)-u(y))(\phi(x)-\phi(y))}{|x-y|^{n+2s}} d x d y=\int_\Omega \frac{f\phi}{u^\gamma} d x  \quad \forall \phi \in C_c^{1}(\Omega),
\end{equation*}
The following result slightly improves the case $\gamma>1$, see [\citealp{garain}, Theorem 2.15] for $p=2$.
\begin{theorem}
    Let $\Omega\subset \mathbb{R}^n$ be open, bounded with $C^1$ boundary and $\exists\, f_0>0$, a constant such that $f \geq f_0>0$ a.e. in $ \Omega$.
If $1<\gamma<\frac{3 m-1}{m+1}$, then $\exists\,u\in W_{\operatorname{loc}}^{1,2}(\Omega)$ satisfying \cref{prob3} with $u^\alpha \in W_0^{1,2}(\Omega),\forall\alpha \in\left(\frac{(m+1)(\gamma+1)}{4 m}, \frac{\gamma+1}{2}\right]$.
\end{theorem} 
\begin{remark}
    Clearly if $1<\gamma<\frac{3 m-1}{m+1}$, we have $\frac{(m+1)(\gamma+1)}{4 m}<1<\frac{\gamma+1}{2}$ and $\alpha=1$ can be chosen in the previous theorem to obtain $u\in W_0^{1,2}(\Omega)$. Therefore the hypothesis on $f$ allows us to obtain a solution $u\in W^{1,2}_0(\Omega)$ for all $m>1$ and $1<\gamma<\frac{3 m-1}{m+1}$, earlier which was restricted for $\gamma\leq 1$ only.
\end{remark}
\begin{proof}
We outline the idea of the proof only. Let $u_k$ be the unique positive weak solution to the approximating problem (see [\citealp{garain}, Lemma 3.2])
\begin{equation}{\label{approximated6}}
\begin{array}{c}
-\Delta u_k+(-\Delta)^s u_k=\frac{f_k}{(u_k+\frac{1}{k})^{\gamma}}\text { in } \Omega, \smallskip\\u_k>0 \text{ in } \Omega,\quad u_k=0  \text { in }\mathbb{R}^n \backslash \Omega;
\end{array}
\end{equation}
where $f_k(x)=\min \{f(x), k\}$. As in Step 2 of the proof of \cref{mainth3}, one can get that the function $z_k(x)=\left(C_0\phi_1(x)+\frac{1}{k^{\frac{(\gamma+1)}{2}}}\right)^{\frac{2}{(\gamma+1)}}-\frac{1}{k}, $ is a subsolution to \cref{approximated6} (as $f\geq f_0$ a.e. in $\Omega$) for sufficiently small $C_0\equiv C_0(f_0,\gamma)$.\smallskip\\ Thus by \cref{subsuper},
\begin{equation}{\label{finalbdd}}
u_k \geq\left(C_0 \phi_1+\frac{1}{k^{(\gamma+1) / 2}}\right)^{2 /(\gamma+1)}-\frac{1}{k} .\end{equation}
Fixing $\alpha>\frac{(m+1)(\gamma+1)}{4 m}$ and $\theta=2 \alpha-1>\frac{(m+1) \gamma+1-m}{2 m}$, one can take $\phi=\left(u_k+\frac{1}{k}\right)^\theta-\left(\frac{1}{k}\right)^\theta$ as a test function in \cref{approximated6} to obtain similarly like \cref{mainth3},\begin{equation*}
\theta \int_{\Omega}\left(u_k+\frac{1}{k}\right)^{\theta-1}\left|\nabla u_k\right|^2 dx\leq \int_{\Omega} f_k \frac{\left(u_k+\frac{1}{k}\right)^\theta-\left(\frac{1}{k}\right)^\theta}{\left(u_k+\frac{1}{k}\right)^\gamma}dx,
\end{equation*}
and hence, by Hölder inequality,
\begin{equation*}
    \frac{4 \theta }{(\theta+1)^2} \int_{\Omega}\left|\nabla\left(\left(u_k+\frac{1}{k}\right)^\alpha-\left(\frac{1}{k}\right)^\alpha\right)\right|^2 dx \leq \int_{\Omega} f_k\left(u_k+\frac{1}{k}\right)^{\theta-\gamma} dx\leq\|f\|_{L^m}\left(\int_{\Omega}\left(u_k+\frac{1}{k}\right)^{m^{\prime}(\theta-\gamma)}\right)^{1 / m^{\prime}} dx.
\end{equation*}
As $\alpha \leq \frac{\gamma+1}{2}$, so $\theta-\gamma \leq 0$. Further $ \theta>\frac{(m+1) \gamma+1-m}{2 m}$ and $\Omega$ satisfies interior sphere condition, so by \cref{finalbdd},
\begin{equation*}
    \left(z_k(x)+\frac{1}{k}\right)^{m^{\prime}(\theta-\gamma)} \leq\left(C_0 \phi_1+\frac{1}{k^{(\gamma+1) / 2}}\right)^{\frac{2 m^{\prime}(\theta-\gamma)}{1+\gamma}}\leq\left(C_0 \phi_1\right)^{\frac{2 m^{\prime}(\theta-\gamma)}{1+\gamma}}\in L^1(\Omega).
\end{equation*}
The rest of the proof follows similarly like \cref{mainth3}.\end{proof}

\bibliography{main.bib}

\bibliographystyle{plain}

%\printbibliography

\end{document}